\documentclass[11pt,reqno]{amsart}

\usepackage{amsmath, graphicx}
\usepackage{amsthm}

\usepackage{amsfonts, amssymb}
\usepackage{mathrsfs}

\providecommand{\U}[1]{\protect\rule{.1in}{.1in}}
\numberwithin{equation}{section}

\newtheorem{theorem}{Theorem}[section]
\newtheorem{corollary}[theorem]{Corollary}
\newtheorem{lemma}[theorem]{Lemma}
\newtheorem{proposition}[theorem]{Proposition}

\theoremstyle{definition}
\newtheorem{definition}[theorem]{Definition}
\newtheorem{remark}[theorem]{Remark}
\newtheorem{example}[theorem]{Example}

\title{Rectifiability of Self-contracted curves \\in the euclidean space and applications}

\author{A. Daniilidis, G. David, E. Durand-Cartagena, A. Lemenant }



\begin{document}


\maketitle

\begin{abstract}It is hereby established that, in
Euclidean spaces of finite dimension, bounded self-contracted curves
have finite length. This extends the main result of \cite{DLS}
concerning \emph{continuous planar} self-contracted curves to any
dimension, and dispenses entirely with the continuity requirement.
The proof borrows heavily from a geometric idea of \cite{MP}
employed for the study of regular enough curves, and can be seen as
a nonsmooth adaptation of the latter, albeit a nontrivial one.
Applications to continuous and discrete dynamical systems are
discussed: continuous self-contracted curves appear as generalized
solutions of nonsmooth convex foliation systems, recovering a hidden
regularity after reparameterization, as consequence of our main
result. In the discrete case, proximal sequences (obtained through
implicit discretization of a gradient system) give rise to polygonal
self-contracted curves. This yields a straightforward proof for the
convergence of the exact proximal algorithm, under any choice of
parameters.
\end{abstract}

\tableofcontents
\newpage

\section{Introduction}

\subsection{Motivation and state-of-the-art}

Self-contracted curves were introduced in
\cite[Definition~1.2.]{DLS} to provide a unified framework for the
study of convex and quasiconvex gradient dynamical systems. Given a
possibly unbounded interval $I$ of ${\mathbb{R}}$, a map
$\gamma:I\rightarrow{\mathbb{R}}^{n}$ is called
\emph{self-contracted curve}, if for every $[a,b]\subset I$, the
real-valued function
\[
t\in\lbrack a,b]\mapsto d(\gamma(t),\gamma(b))
\]
is non-increasing. This notion is purely metric and does not require any prior
smoothness/continuity assumption on $\gamma$.

So far, self-contracted curves are considered in a Euclidean
framework. In particular, given a smooth function
$f:{\mathbb{R}}^{n}\rightarrow{\mathbb{R}}$, any solution $\gamma$
of the gradient system
\begin{equation}
\left\{
\begin{array}
[c]{l}
\gamma^{\prime}(t)=-\nabla f(\gamma(t))\qquad t>0,\\[4pt]
\gamma(0)=x_{0}\in{\mathbb{R}}^{n}
\end{array}
\right.  \label{diffeq}
\end{equation}
is a (smooth) self-contracted curve, provided that $f$ is \emph{quasiconvex}, that
is, its sublevel sets
\[
\lbrack f\leq\beta]:=\{x\in{\mathbb{R}}^{n}:f(x)\leq\beta\}\quad(\beta
\in{\mathbb{R}})
\]
are convex subsets of ${\mathbb{R}}^{n}$
\cite[Proposition~6.2]{DLS}. Self-contracted curves also appear in
subgradient systems, defined by a (nonsmooth) \emph{convex}
function~$f$, see \cite[Proposition~6.4]{DLS}. In this case, the
first equation in~\eqref{diffeq} becomes the following differential
inclusion
\begin{equation}
\gamma^{\prime}(t)\in-\partial f(\gamma(t))\quad\mathrm{a.e.,}
\label{subgrad}
\end{equation}
where $\mathrm{a.e.}$ stands for ``almost everywhere'' and the
solutions are absolutely continuous curves (see~\cite{Brezis} for a
general theory).

A central question of the asymptotic theory of a general gradient
dynamical system of the form \eqref{diffeq} is whether or not
bounded orbits are of finite length, which if true, yields in
particular their convergence. This property fails for $C^{\infty}$
smooth functions (\cite[p.~12]{PalDem82}), but holds for analytic
gradient systems (\cite{Loja63}), or more generally, for systems
defined by an o-minimal (tame) function (\cite{Kurdyka98}). In these
cases, a concrete estimation of the length of (sub)gradient curves
is obtained in terms of a so-called \L ojasiewicz type inequality,
intrinsically linked to the potential $f$, see \cite{BDLM} for a
survey, which also includes extensions of the theory to subgradient
systems. Notice however that convex functions fail to satisfy such
inequality, see \cite[Section~4.3]{BDLM} for a counterexample.

In \cite{MP} a certain class of Lipschitz curves has been introduced
(with no specific name)  to capture the behaviour of orbits of
quasiconvex potentials. Unlike self-contracted curves, this notion
makes sense only in Euclidean spaces (it uses orthogonality) and
requires the curve to be Lipschitz continuous. The main result of
\cite{MP} asserts that the length of such  curves is bounded by the
mean width of their convex hull, a fortiori, by the mean width of
any convex set containing the curve. Recently, the authors of
\cite{GS} extended the result of \cite{MP} to 2-dimensional surfaces
of constant curvature, naming these curves as (G)-orbits. We shall
prefer to call these curves \emph{self-expanding}, see
Definition~\ref{Def_EC} below. The choice of this terminology will
become clear in Section~\ref{sectionselfself}.

Formally the result, as announced in \cite[IX]{MP}, applies only
under the \emph{prior} requirement that the length of the curve is
finite, since the proof is given for self-expanded curves
parameterized by the arc-length parametrization in a compact
interval. This restriction is removed in Section~\ref{Sub21} via a
simple continuity argument (see Corollary~\ref{main}). As a result,
the smooth orbits of \eqref{diffeq} for a quasiconvex potential, as
well as the absolutely continuous orbits of \eqref{subgrad} for a
convex one have finite length. In both cases the bound depends only
on the diameter of the initial sublevel set.

The work \cite{MP} was unknown to the authors of \cite{DLS}, who
tackled the same problem in terms of the aforementioned notion of
self-contracted curve. The definition of self-contracted curve does
not require any regularity neither on the space nor on the curve. In
particular, such curves need not be continuous, and
differentiability may a priori fail at each point. The main result
of \cite{DLS} shows that bounded continuous planar self-contracted
curves have finite length (\cite[Theorem~1.3]{DLS}). This has been
used to deduce that in $\mathbb{R}^{2}$, smooth orbits of
quasiconvex systems (respectively, absolutely continuous orbits of
nonsmooth convex systems) have finite length. As we saw before, this
conclusion essentially derives from \cite[IX]{MP} for any dimension
(see comments above). Notice however that the main result of
\cite{DLS} cannot be deduced from \cite{MP}, but can only be
compared in retrospect, once rectifiability is established.

On the other hand, not completely surprisingly, \emph{Lipschitz
continuous} self-contracted curves and self-expanded curves turn out to
be intimately related and can be obtained one from the other by
means of an adequate reparameterization, reversing orientations (see
Lemma~\ref{reverse}). Recall however that both rectifiability and
Lipschitz continuity of the curve are prior requirements for the
definition of a self-expanded curve, while they are neither
requirements nor obvious consequences of the definition of a
self-contracted curve.

\subsection{Contributions of this work}

In this work we establish rectifiability of any self-contracted
curve in $\mathbb{R}^{n}$, by extending the result of \cite{DLS} to
any dimension, and to possibly discontinuous curves. This is done by
adapting the geometrical idea of \cite{MP} to the class of
self-contracted curves. This nonsmooth adaptation is natural but
quite involved. Nonsmooth variations of the mean width of the closed
convex hull of the curve are again used to control the increase of
its length, but no prior continuity on the parametrization is
required and rectifiability is now part of the conclusions. Namely,
setting $\Gamma=\gamma(I)$ (the image of the curve in
${\mathbb{R}}^{n}$) and denoting by $\ell(\gamma)$ its length, we
establish the following result (see Theorem~\ref{main1})\,:

\begin{itemize}
\item Every self-contracted curve is rectifiable and satisfies the relation
$\ell(\gamma)\leq C\,\mathrm{diam}(\Gamma)$, where $C>0$ depends only on the dimension.
\end{itemize}

In case of continuous curves, the above result allows to consider a
Lipschitz reparameterization defined by the arc-length, see details
in Section~\ref{sec32}. This leads to the following conclusion:

\begin{itemize}
\item If $\gamma$ is a \emph{continuous} self-contracted curve and
$\Gamma=\gamma(I)$ is bounded, then $\Gamma$ is also the image of some
(Lipschitz) self-expanded curve.
\end{itemize}

In particular, the sets of all possible images of continuous self-contracted
curves and of self-expanded curves coincide. Still, the set of images of all
self-contracted curves is much larger (its elements are not connected in
general).

In the last two sections, two new applications of self-contracted
curves are considered. In Section~\ref{Section_foliations} we broaden
the framework of dynamical systems to encompass nonsmooth convex
foliation systems, with merely continuous generalized orbits. Limits
of backward secants remedy the absence of differentiability, leading to
a consistent notion of generalized solution in the sense of
nonsmooth analysis (Definition~\ref{Definition_orbit-foliation}). In
Theorem~\ref{Theorem_orbit-foliation} we show that these generalized
solutions are self-contracted curves, thus of finite length; in view
of the aforementioned result, they can also be obtained as
\textquotedblleft classical\textquotedblright\ solutions through an
adequate Lipschitz reparameterization. On the other hand, $C^{1}$
smooth convex foliation orbits enjoy a stronger property, the
so-called \emph{strong self-contractedness}, see
Definition~\ref{def-strong-sc} and Corollary~\ref{Cor_strongSC}.
Concerning this latter class, we establish in Section~\ref{sec42}
the following approximation result, with respect to the Hausdorff
distance, see Proposition~\ref{Prop_polyg}.

\begin{itemize}
\item Every $C^{1}$-smooth strongly self-contracted curve is a limit of
polygonal self-contracted curves.
\end{itemize}

Finally, in Section~\ref{Section_proximal} we provide an elegant
application of the notion of self-contracted curve in a different
framework, that of discrete systems. In particular we establish the
following result (Theorem~\ref{thm_prox})\,:

\begin{itemize}
\item Let $f$ be any convex function, bounded from below. Then the exact
proximal algorithm gives rise to a self-contracted polygonal curve.
\end{itemize}

In view of our main result, we obtain a straightforward proof of the
convergence of the proximal algorithm. The bound over the length of
the polygonal curve yields a sharp estimation on the rate of
convergence, which appears to be entirely new. Notice that the
convergence result is independent of the choice of the parameters.

\subsection{Acknowledgements}

The first author acknowledges support of the grant MTM2011-29064-C01
(Spain) and thanks Jerome Bolte and Joel Benoist for useful
discussions. The third author is partially supported by grant
MTM2009-07848 (Spain). The second and fourth authors are partially
supported by the ANR project GEOMETRYA (France). Part of this work
has been realized during a research stay of the third author at the
Universit\'e Paris Diderot (Paris 7) and Laboratory Jacques Louis
Lions. The stay was supported by the program ``Research in Paris''
offered by the Ville de Paris (Mairie de Paris). This author thanks
the host institution and Ville de Paris for its hospitality.

\section{Notation and preliminaries}

Let $({\mathbb{R}}^{n},d,\mathscr{L}^{n})$ denote the
$n$-dimensional Euclidean space endowed with the Euclidean distance
$d(x,y)=\Vert x-y\Vert$, the scalar product
$\langle\cdot,\cdot\rangle$, and the Lebesgue measure
$\mathscr{L}^{n}$. We denote by $B(x,r)$ (respectively,
$\overline{B}(x,r)$) the open (respectively, closed) ball of radius
$r>0$ and center $x\in {\mathbb{R}}^{n}$. If $A$ is a nonempty
subset of ${\mathbb{R}}^{n}$, we denote by $\sharp A$ its
\emph{cardinality}, by $\mathrm{conv\,}(A)$ its \emph{convex hull}
and by $\text{diam}\,A:=\sup\,\{d(x,y):x,y\in A\}$ its
\emph{diameter}. We also denote by $\mathrm{int}(A)$, $\overline{A}$
and $\partial A$ the \emph{interior}, the \emph{closure} and
respectively, the \emph{boundary} of the set $A$.

Let now $K$ be a nonempty closed convex subset of ${\mathbb{R}}^{n}$
and $u_{0}\in K$. The normal cone $N_{K}(u_{0})$ is defined as
follows:
\begin{equation}
N_{K}(u_{0})=\{v\in\mathbb{R}^{n}:\langle v,u-u_{0}\rangle\leq0,\forall u\in
K\}. \label{NKu0}
\end{equation}
Notice that $N_{K}(u_{0})$ is always a closed convex cone. Notice
further that $u_0\in K$ is the projection onto $K$ of all elements
of the form $u_0+tv$, where $t\geq 0$ and $v\in N_K(u_0)$.

The \emph{Hausdorff distance} between two nonempty closed subsets
$K_{1},K_{2}$ of ${\mathbb{R}}^{n}$ is given by the formula
\[
d_{H}(K_{1},K_{2}):=\inf\bigg\lbrace\varepsilon>0\,:\,K_{2}\subset
\bigcup_{z\in K_{1}}B(z,\varepsilon)\text{ and }K_{1}\subset\bigcup_{z\in
K_{2}}B(z,\varepsilon)\bigg\rbrace.
\]
A sequence of closed sets $\{K_{j}\}_{j}$ in ${\mathbb{R}}^{n}$ is said to
\emph{converge with respect to the Hausdorff distance} to a closed set
$K\subset{\mathbb{R}}^{n}$ if $d_{H}(K_{j},K)\rightarrow0$ as $j\rightarrow
\infty$.

Throughout the manuscript, $I$ will denote a possibly unbounded
interval of $\mathbb{R}$. In this work, a usual choice for the
interval will be $I=[0,T_{\infty})$ where
$T_{\infty}\in\mathbb{R\cup\{+\infty\}}$. A mapping
$\gamma:I\rightarrow {\mathbb{R}}^{n}$ is referred in the sequel as
a \emph{curve}. Although the usual definition of a curve comes along
with continuity and injectivity requirements for the map $\gamma$,
we do not make these prior assumptions here. By the term
\emph{continuous} (respectively, \emph{absolutely continuous},
\emph{Lipschitz}, \emph{smooth}) curve we shall refer to the
corresponding properties of the mapping
$\gamma:I\rightarrow{\mathbb{R}}^{n}$. A curve $\gamma$ is said to
be \emph{bounded} if its image, denoted by $\Gamma=\gamma(I)$, is a
bounded set of ${\mathbb{R}}^{n}$.

The \emph{length} of a curve $\gamma:I\rightarrow{\mathbb{R}}^{n}$
is defined as
\begin{equation}
\ell(\gamma):=\sup\Big\{\sum_{i=0}^{m-1}d(\gamma(t_{i}),\gamma(t_{i+1}
))\Big\},\label{length-formula}
\end{equation}
where the supremum is taken over all finite increasing sequences
$t_{0}<t_{1}<\cdots<t_{m}$ that lie in the interval $I$. Notice that
$\ell(\gamma)$ corresponds to the total variation of the function
$\gamma:I\rightarrow {\mathbb{R}}^{n}$. Let us mention for completeness that
the length $\ell(\gamma)$ of a continuous injective curve $\gamma$
is equal to the unidimensional Hausdorff measure
$\mathcal{H}^{1}(\Gamma)$ of its image, see e.g. \cite[Th.2.6.2]{BBI}, but it is in general greater
for noncontinuous curves. In particular we emphasis that for a piecewise continuous curve, the quantity $\ell(\gamma)$ is strictly greater than the sum of the lengths of each pieces but  we still call $\ell(\gamma)$ the length of $\gamma$. A curve is called
\emph{rectifiable}, if  it has locally bounded length.

Let us finally define the \emph{width} of a (nonempty) convex subset $K$
of $\mathbb{R}^{n}$ at the direction $u\in\mathbb{S}^{n-1}$ as being the length of
its orthogonal projection $P_{u}(K)$ on the 1--dimensional space $\mathbb{R}u$
generated by $u$. The following definition will play a key role in our main result.

\begin{definition}
[Mean width]\label{meanwidth}The \emph{mean width} of a nonempty convex set
$K\subset\mathbb{R}^{n}$ is given by the formula
\[
W(K)=\frac{1}{\sigma_{n}}\int_{{\mathbb{S}}^{n-1}}\mathscr{L}^{1}
(P_{u}(K))du,
\]
where $du$ denotes the standard  volume form on
$\mathbb{S}^{n-1}$, and
\[
\sigma_{n}=\int_{{\mathbb{S}}^{n-1}}du=\frac{n\pi^{n/2}}{\Gamma(\frac{n}{2}+1)}.
\]

\end{definition}

\subsection{Self-expanded curves}

\label{Sub21} Let us now recall from \cite{MP} the definition and
the basic properties of a favorable class of Lipschitz curves, which
has been studied thereby without a specific name. In the sequel we
call these curves \emph{self-expanded}.

\begin{definition}
[Self-expanded curve]\label{Def_EC}A Lipschitz curve
$\gamma:I\rightarrow {\mathbb{R}}^{n}$ is called \emph{self-expanded
curve} if for every $t\in I$ such that $\gamma^{\prime}(t)$ exists,
we have that
$\big\langle\gamma^{\prime}(t),\gamma(t)-\gamma(u)\big\rangle\geq0$
for all $u\in I$ such that $u\leq t$.
\end{definition}

In \cite[3.IX.]{MP} the following result has been established concerning
self-expanded curves.

\begin{theorem}
[{\cite[3.IX.]{MP}}]\label{boundedec} Let $\gamma: I
\rightarrow{\mathbb{R}}^{n}$ be a self-expanded curve of finite
length. Then there exists a constant $C>0$ depending only on the
dimension $n$ such that
\begin{equation}
\ell(\gamma)\leq C\,\text{\textrm{diam}}\,K, \label{mp-1}
\end{equation}
where $K$ is any compact set containing $\Gamma=\gamma(I)$.
\end{theorem}

Notice that, formally, the above result requires the curve to have finite
length. Nevertheless, the following limiting argument allows to obtain a more
general conclusion for bounded self-expanded curves.

\begin{corollary}
[Bounded self-expanded curves have finite length]\label{main}Every bounded
self-expanded curve $\gamma:I\rightarrow{\mathbb{R}}^{n}$ has finite length and \eqref{mp-1} holds.
\end{corollary}

\begin{proof} Since self-expanded curves are rectifiable by definition, and because  reparameterizing $\gamma$ does not change the statement, we
may assume that $\gamma$ is parameterized by its arc-length on
$I=[0,\ell (\gamma))$. Notice though that in principle
$\ell(\gamma)$ might be infinite. Our aim is precisely to show that
this is not the case. Indeed, let $K$ be a compact set containing
$\gamma(I)$ and let $\{L_{n}\}_{n}$ be an increasing sequence of
real numbers converging to
$\ell(\gamma)\in\mathbb{R\cup\{+\infty\}}$. Applying
Theorem~\ref{boundedec} for the curve
$\gamma_{n}:[0,L_{n}]\rightarrow{\mathbb{R}}^{n}$ (restriction of
$\gamma$ to $[0,L_{n}]$), we obtain
\[
L_{n}=\ell(\gamma_{n})\leq C\,\mathrm{diam}\,K,\quad\text{for all }n\geq1.
\]
This shows that $\ell(\gamma)=\lim_{n\rightarrow+\infty}L_{n}$ is
bounded and satisfies the same estimate.
\end{proof}

\subsection{Self-contracted versus self-expanded curves}
\label{sectionselfself}

The aim of this section is to prove that Lipschitz continuous
self-contracted curves and self-expanded curves give rise to the same
images. Moreover, each of these curves can be obtained from the
other upon reparameterization (inverting orientation). As a
byproduct, bounded Lipschitz self-contracted curves in
${\mathbb{R}}^{n}$ have finite length.

Let us recall the definition of a self-contracted curve (see
\cite[Definition 1.2]{DLS}).

\begin{definition}
[Self-contracted curve]\label{defsc}\thinspace A curve
$\gamma:I\rightarrow {\mathbb{R}}^{n}$ is called \em self-contracted\em, if
for every $t_{1}\leq t_{2}\leq t_{3}$ in $I$ we have
\begin{equation}
d(\gamma(t_{1}),\gamma(t_{3}))\geq d(\gamma(t_{2}),\gamma(t_{3})).
\label{SC-ineq}
\end{equation}
In other words, the function $t\mapsto d(\gamma(t),\gamma(t_{3}))$ is
nonincreasing on $I\cap(-\infty,t_{3}].$
\end{definition}

\begin{remark}
\label{scmetric}(i) As we already said before, the definition of
self-contracted curve can be given in any metric space and does not
require any regularity of the curve, such as continuity or
differentiability. Notice moreover that if
$\gamma(t_{1})=\gamma(t_{3})$ in \eqref{SC-ineq} above, then
$\gamma(t)=\gamma(t_{1})$ for $t_{1}\leq t\leq t_{2}$; thus if
$\gamma$ is not locally stationary, then it is
injective.\smallskip\newline (ii) It has been proved in
\cite[Proposition~2.2]{DLS} that if $\gamma$ is self-contracted and
bounded, and $I=[0,T_{\infty})$ with $T_{\infty
}\in\mathbb{R}^+\cup\{+\infty\}$, then $\gamma$ converges to some
point $\gamma_{\infty}\in{\mathbb{R}}^{n}$ as $t\rightarrow
T_{\infty}$. (Notice that this conclusion follows also from our main
result Theorem~\ref{main1}.) Consequently the curve $\gamma$ can be
extended to $\bar {I}=[0,T_{\infty}]$. In particular, if $\gamma$ is
continuous, then denoting by $\Gamma=\gamma(I)$ the image of
$\gamma$, it follows that the set
\[
\bar{\Gamma}=\gamma(I)\cup\{\gamma_{\infty}\}=\gamma(\bar I)
\]
is a compact subset of ${\mathbb{R}}^{n}$.
\end{remark}

\begin{lemma}
[Property of a differentiability point]\label{halfspace} Let $\gamma
:I\rightarrow{\mathbb{R}}^{n}$ be a self-contracted curve and let
$t$ be a point of differentiability of $\gamma$. Then
\[
\langle\gamma^{\prime}(t),\gamma(u)-\gamma(t)\rangle\geq0\ \hbox{
for all $u\in I$ such that $u > t$.}
\]

\end{lemma}

\begin{proof} Assume that $\gamma$ is differentiable at $t\in
I$, and write $\gamma(t+s)=\gamma(t)+s\gamma^{\prime}(t)+o(s)$, with
$\lim_{s\rightarrow 0}s^{-1}o(s)=0$. Let $u\in I$ be such that
$u>t$, take $s$ such that $0< s <u-t$ and apply \eqref{SC-ineq} for
$t_{1}=t$, $t_{2}=t+s$ and $t_{3}=u$. We deduce that
\[
\Vert\gamma(t)-\gamma(u)\Vert\geq\Vert\gamma(t+s)-\gamma(u)\Vert.
\]
Since
$\gamma(t+s)-\gamma(u)=\gamma(t)-\gamma(u)+s\gamma^{\prime}(t)+o(s)$,
substituting this in the above inequality and squaring yields
\[
\begin{aligned}
0 &\geq \|\gamma(t+s)-\gamma(u)\|^2 - \|\gamma(t)-\gamma(u)\|^2 \cr
& = 2 \langle s\gamma'(t) + o(s),\gamma(t)-\gamma(u)\rangle +
\|s\gamma'(s)+o(s)\|^2 \cr & = 2 s\langle \gamma'(t) ,\gamma(t)-\gamma(u)\rangle
+ o(s).
\end{aligned}
\]
Dividing by $s$, and taking the limit as $s$ tends to $0^{+}$ we get the
desired result.
\end{proof}

Given a curve $\gamma:I\mapsto{\mathbb{R}}^{n}$, we denote by
$I^{-}=-I=\{-t\,;\,t\in I\}$ the opposite interval, and define the
reverse parametrization $\gamma^{-}:I^{-}\mapsto{\mathbb{R}}^{n}$ of
$\gamma$ by $\gamma^{-}(t)=\gamma(-t)$ for $t\in I^{-}$.

\begin{lemma}
[Lipschitz self-contracted versus self-expanded curves]\label{reverse} Let
$\gamma:I\rightarrow{\mathbb{R}}^{n}$ be a Lipschitz curve. Then $\gamma$ is
self-contracted if and only if $\gamma^{-}$ is a self-expanded curve.
\end{lemma}

\begin{proof} If $\gamma:I\rightarrow{\mathbb{R}}^{n}$ is
Lipschitz and self-contracted, then Lemma~\ref{halfspace} applies,
yielding directly that $\gamma^{-}$ is a self-expanded curve.
Conversely, suppose that $\gamma^{-}$ is a self-expanded curve. This
means that $\gamma$ is Lipschitz and (after reversing the
orientation) that
\begin{equation}
\langle\gamma^{\prime}(t),\gamma(u)-\gamma(t)\rangle\geq0\ \hbox{
for $u\in I$ such that $u > t$} \label{ant1}
\end{equation}
whenever $\gamma^{\prime}(t)$ exists and is different from $0$.

\noindent Fix now any $t_{3}\in I$ and define the  function
\[
f(t)=\frac{1}{2}\Vert\gamma(t)-\gamma(t_{3})\Vert^{2},\quad\text{for all } t\in
I\text{, }t\leq t_{3}.
\]
By Rademacher's theorem the Lipschitz continuous functions $\gamma$ is differentiable $\mathscr{L}^{1}$-almost everywhere, and so $f$ is too, and  $f^{\prime
}(t)=\langle\gamma^{\prime}(t),\gamma(t)-\gamma(t_{3})\rangle$ for almost all
$t\in I\cap(-\infty,t_{3})$.

If $t<t_{3}$ and $\gamma^{\prime}(t)\neq0$ then \eqref{ant1} above (for
$u=t_{3}$) yields that $f^{\prime}(t)\leq0$. Otherwise, $f^{\prime}(t)=0$. It
follows that the Lipschitz function $f$ satisfies $f^{\prime}(t)\leq0$
$\mathscr{L}^{1}$-almost everywhere, thus it is nondecreasing. This
establishes \eqref{SC-ineq}. Since $t_{3}$ has been chosen arbitrarily, the
proof is complete.
\end{proof}

\section{Rectifiability of self-contracted curves}

\label{section_RECT} In \cite[Theorem 1.3]{DLS} it has been
established that bounded self-contracted continuous planar curves
$\gamma:[0,+\infty )\rightarrow\mathbb{R}^{2}$ have finite length.
In this section we improve this result by dropping the continuity
assumption, and we extend it to any dimension. Precisely, we
establish that the length of any self-contracted curve
$\gamma:I\rightarrow{\mathbb{R}}^{n}$ lying inside a compact set is
bounded by a quantity depending only on the dimension of the space
and the diameter of the compact set, see the forthcoming
Theorem~\ref{main1}.

\subsection{Proof of the main result}

The proof makes use of the following technical facts.

\begin{lemma}
[Saturating the sphere]\label{counting} Let $\Sigma\subseteq\mathbb{S}^{n-1}$
be such that $\langle x,y\rangle\leq1/2$ for all $x,y\in\Sigma,x\not =y$. Then
$\Sigma$ is finite and $\sharp\Sigma\leq3^{n}$.
\end{lemma}

\begin{proof} For any $x,y\in\Sigma,x\not =y$, we have $\Vert
x-y\Vert^{2}=2-2\langle x,y\rangle\geq1$. Therefore, the open balls
$\{B(x,1/2)\}_{x\in\Sigma}$ are disjoint and they are all contained in the
ball $B(0,\frac{3}{2})$. Set $\omega_{n}=\mathscr{L}^{n}(B(0,1))$ (the measure of the unit
ball); then
\[
(\sharp\Sigma)\;\omega_{n}\left(  \frac{1}{2}\right)
^{n}=\mathscr{L}^{n}\Big( \bigcup _{x\in\Sigma}B(x,1/2)\Big)
\leq\mathscr{L}^{n}\big(B(0,3/2)\big)=\omega_{n}\left(
\frac{3}{2}\right) ^{n}
\]
and so $\sharp\Sigma\leq3^{n}$.
\end{proof}

\begin{lemma}
[Hemisphere lemma]\label{caratheodory} Let $\Sigma\subset\mathbb{S}^{n-1}$
be a set satisfying
\begin{equation}
\langle x,y\rangle\geq-\left(  \frac{1}{3}\right)  ^{n+1}\ \hbox{
for all }x,y\in\Sigma. \label{hypc}
\end{equation}
Then there exists $\zeta\in\mathbb{S}^{n-1}$ such that
\[
\langle\zeta,x\rangle\geq\left(  \frac{1}{3}\right)  ^{2n+1}\ \hbox{ for all }x\in
\Sigma.
\]

\end{lemma}

\begin{proof} Let $\{x_{i}\}_{i\in I}$ be a family of points in
$\Sigma$, maximal with respect to the property that
\begin{equation}
\langle x_{i},x_{j}\rangle\leq\frac{1}{2}\ \hbox{ for all }i\not
=j\text{ in }I\,. \label{presqueorth}
\end{equation}
(Notice that such a family can easily be constructed by induction.)
Applying Lemma~\ref{counting} we deduce that $m:=\sharp I\leq3^{n}$.
Set
\[
v:=\sum_{i\in I}x_{i}
\]
and let $y\in\Sigma$ be an arbitrary point. If $y$ does not belong to the
family $\{x_{i}\}_{i\in I}$ then by maximality of the latter, there exists
some $i_{0}$ such that $\langle x_{i_{0}},y\rangle>\frac{1}{2}$. If on the
contrary, $y=x_{i}$ for some $i$, then we take $i_{0}=i$. In view of
\eqref{hypc} we obtain
\begin{equation}
\langle v,y\rangle=\langle x_{i_{0}},y\rangle+~\sum_{i\not
=i_{0}}\langle x_{i},y\rangle\geq\frac{1}{2}-(m-1)\left(
\frac{1}{3}\right)  ^{n+1}, \label{restim}
\end{equation}
whence, recalling that $m\leq3^{n}$, we get $\langle
y,v\rangle\geq3^{-(n+1)}$. This shows in particular that $v\neq0$.
Let us now compute $\Vert v\Vert$. We write
\begin{equation}
\Vert v\Vert^{2}=\Big\|\sum_{i\in I}x_{i}\Big\|^{2}=\sum_{i\in
I}\Vert x_{i}\Vert^{2}+E=m+E \label{norm}
\end{equation}
where
\[
E=\sum_{i\in I}\sum_{j\not =i}\langle x_{i},x_{j}\rangle.
\]
Using \eqref{presqueorth} we deduce that $|E|\leq m(m-1)/2$. Therefore
\eqref{norm} yields
\[
\Vert v\Vert^{2}\leq
m+\frac{m(m-1)}{2}=\frac{m}{2}(m+1)\leq\frac{3^{n}}
{2}(3^{n}+1)\leq3^{2n}.
\]
So dividing by $\Vert v\Vert$ in \eqref{restim} and setting
\[
\zeta:=\frac{v}{\Vert v\Vert}\in\mathbb{S}^{n-1}
\]
we obtain
\[
\langle\zeta,y\rangle=\frac{1}{\Vert v\Vert}\langle
v,y\rangle\geq\left( \frac{1}{3}\right)  ^{n}\left(
\frac{1}{3}\right)  ^{n+1}=\left(  \frac{1} {3}\right)  ^{2n+1}.
\]

Since $y$ is arbitrary in $\Sigma,$ the proof is complete.
\end{proof}

\bigskip

We are now ready to prove the main result of this section.

\begin{theorem}
\label{main1} Let $\gamma:I\rightarrow{\mathbb{R}}^{n}$ be a self-contracted
curve. Then there exists a constant $C_{n}$ (depending only on the dimension
$n$) such that
\begin{equation}
\ell(\gamma)\leq C_{n}W(K), \label{length}
\end{equation}
where $K$ is the closed convex hull of $\gamma(I).$ In particular, bounded
self-contracted curves have finite length.
\end{theorem}

\begin{proof} The result holds vacuously for unbounded curves (both
left-hand and right-hand side of \eqref{length} are equal to $+\infty$).
Therefore, we focus our attention on bounded self-contracted curves and assume that
$K$ is compact. We may also clearly assume  that $n\geq2$ (the result is trivial in
the one-dimensional case).

In the sequel, we denote by $\Gamma=\gamma(I)$ the image of such
curve. The set $\Gamma$ inherits from $I$ a total order as follows:
for $x,y\in\Gamma$ we say that \textquotedblleft$x$ is before $y$
\textquotedblright\ and denote $x\preceq y$, if there exist
$t_{1},t_{2}\in I,$ $t_{1}\leq t_{2}$ and $\gamma(t_{1})=x,$
$\gamma(t_{2})=y$. If $x\preceq y$ and $x\neq y,$ then the intervals
$\gamma^{-1}(x)$ and $\gamma^{-1}(y)$ do not meet, and for any
$t_{1}\in\gamma^{-1}(x)$, $t_{2}\in\gamma^{-1}(y)$ we have
$t_{1}<t_{2}$. In this case we say that \textquotedblleft$x$ is
strictly before $y$\textquotedblright\ and we denote $x\prec y$. For
$x\in\Gamma$ we set
\[
\Gamma(x):=\{y\in\Gamma:x\preceq y\}
\]
(the piece of curve after $x$) and denote by $\Omega(x)$ the closed
convex hull of $\Gamma(x)$.

\medskip

\textit{Claim~1}. To establish \eqref{length} it suffices to find a
positive constant $\varepsilon=\varepsilon(n)$, depending only on
the dimension $n$, such that for any two points
$x,x^{\prime}\in\Gamma$ with $x^{\prime}\preceq x$ it holds
\begin{equation}
W(\Omega(x))+\varepsilon\Vert x-x^{\prime}\Vert\leq
W(\Omega(x^{\prime})). \label{estim2}
\end{equation}

\textit{Proof of Claim~1}. Let us see how we can deduce
Theorem~\ref{main1} from the above. To this end, let
$t_{0}<t_{1}\ldots<t_{m}$ be any increasing sequence in $I$, and set
$x_{i}=\gamma(t_{i})$. If \eqref{estim2} holds, then
\[
\begin{aligned} \sum_{i=0}^{m-1}\|\gamma(t_{i+1})- \gamma(t_i)\|
&= \sum_{i=0}^{m-1}\|x_{i+1}- x_{i}\| \leq \frac{1}{\varepsilon}
\sum_{i=0}^{m-1}\big(W(C(x_{i})-W(C(x_{i+1}))\big) \cr& =
\frac{1}{\varepsilon} (W(C(x_{0}))-W(C(x_{m}))) \leq
\frac{1}{\varepsilon} W(C(x_{0})) \leq \frac{1}{\varepsilon} W(K),
\end{aligned}
\]
since the mean width $W(H)$ is a nondecreasing function of $H$ (the variable
$H$ is ordered via the set inclusion). Taking the supremum over all choices of
$t_{0}<t_{1}\ldots<t_{m}$ in $I$ we obtain  \eqref{length} for
$C_{n}=1/\varepsilon$.\hfill$\blacktriangle$

\medskip

Therefore, the theorem will be proved, if we show that \eqref{estim2} holds
for some constant $\varepsilon>0$ which depends only on the dimension. Before
we proceed, we introduce some extra notation. Given $x,x^{\prime}$ in $\Gamma$
with $x^{\prime}\prec x$ we set
\begin{equation}
x_{0}:=\frac{1}{3}x+\frac{2}{3}x^{\prime}\ \hbox{ and }\
v_{0}:=\frac {x^{\prime}-x}{\Vert x^{\prime}-x\Vert}. \label{guy}
\end{equation}
For the sake of drawing pictures, the reader is invited to think
that $v_{0}=e_{1}$ (the first vector of the canonical basis of
$\mathbb{R}^{n}$), see Figure~1. Let us also set
\begin{equation}
\xi_{0}(y)=\frac{y-x_{0}}{\Vert y-x_{0}\Vert}\in\mathbb{S}^{n-1}
,\quad\text{for any }y\in\Gamma(x). \label{guy-1}
\end{equation}
Clearly, $x_{0}$, $v_{0}$ and $\xi_{0}(y)$ depend on the points $x,x^{\prime
},$ while the desired constant $\varepsilon$ does not. To determine this
constant, we shall again transform the problem into another one (see the
forthcoming Claim~2).

 \begin{figure}[h!]
 \centering
  \includegraphics[width=0.5\textwidth]{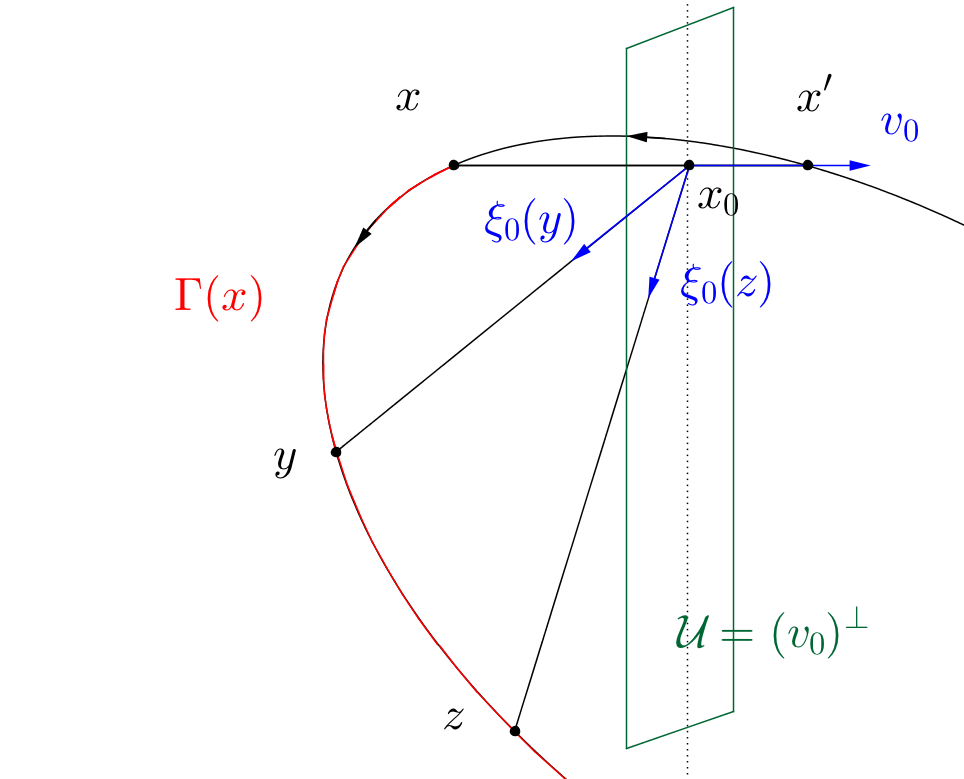}
 \caption{Controlling the tail of a self-contracted curve.}
 \end{figure} 

\textit{Claim~2}. Let us assume that there exists a constant
$0<\delta<2^{-4}$, depending only on the dimension $n$, such that
for all $x,x^{\prime}$ in $\Gamma$ with $x^{\prime}\prec x$ (and for
$x_{0},v_{0}$ defined by \eqref{guy}), there exists
$\bar{v}\in\mathbb{S}^{n-1}\cap B(v_{0},\delta)$ such that
\begin{equation}
\big\langle\overline{v},\xi_{0}(y)\big\rangle\leq-\delta^{2}\quad\text{
for all }y\in\Gamma(x). \label{CL21}
\end{equation}
Then \eqref{estim2} holds true (and consequently \eqref{length} follows).

\textit{Proof of Claim~2}. Assume that such a constant $\delta$ and a
vector $\bar{v}$ exist, so that \eqref{CL21} holds. Set
\[
V=\big\{v\in\mathbb{S}^{n-1}\,;\,\Vert v-\overline{v}\Vert\leq\delta
^{2}\big\}.
\]
Combining with \eqref{guy-1} and \eqref{CL21} we get
\begin{equation}
\langle v,y-x_{0}\rangle\leq0,\quad\text{for all }v\in V\;\text{and
} y\in\Gamma(x). \label{polar}
\end{equation}
Let us first explain intuitively why the above yields
\eqref{estim2}. Indeed, compared with $\Omega(x)$,
$\Omega(x^{\prime})$ has an extra piece coming from the segment
$[x_{0},x^{\prime}].$ This piece is protruding in all directions
$v\in V$ (which are relatively close to $v_{0}$), while
\eqref{polar} bounds uniformly the orthogonal projections of
$\Omega(x)$ onto the lines $\mathbb{R}v$. Therefore, the extra
contribution of the segment $[x_{0},x^{\prime}]$ in
$P_{v}(\Omega(x^{\prime}))$ becomes perceptible and can be
quantified in terms of $\|x_{0}-x^{\prime}\|$, uniformly in $V.$
Since the latter set $V$ has a positive measure, the estimation
\eqref{estim2} follows.

To proceed, observe that $\Omega(x^{\prime})$ contains the convex hull of
$\Omega(x)\cup\lbrack x,x^{\prime}]$, whence
\[
P_{v}(\Omega(x))\subset P_{v}(\Omega(x^{\prime})).
\]
In particular,
\begin{equation}
\mathcal{H}^{1}(P_{v}(\Omega(x)))\leq\mathcal{H}^{1}(P_{v}(\Omega(x^{\prime
})))\ \hbox{ for }v\in\mathbb{S}^{n-1}, \label{es-1}
\end{equation}
where $P_{v}$ denotes the orthogonal projection onto the line
$\mathbb{R}v.$ Let us now equip the latter with the obvious order
(stemming from the identification $\mathbb{R}v\cong\mathbb{R}$) and
let us identify $P_{v}(x_{0})$ with the zero element~$0$. Then
\eqref{polar} says that for all directions $v$ in $V$ we have
\[
\sup P_{v}(\Omega(x))\leq0<P_{v}(x^{\prime})\leq\sup P_{v}(\Omega(x^{\prime
})).
\]
Notice that
\[
\|x_{0}-x^{\prime}\|=\frac{1}{3}\|x-x^{\prime}\|,
\]
and that $V\subset B(v_{0},\delta+\delta^{2})$. Thus for every $v\in
V$ we have
\[
\langle v_{0},v\rangle\geq\langle
v_{0},v_{0}\rangle-\|v_{0}\|\|v_{0}
-v\|\geq1-\delta-\delta^{2}\geq7/8.
\]
This gives a lower bound for the length of the projected segment
$[x_{0},x^{\prime}]$ onto $\mathbb{R}v,$ which coincides, under the above
identification, with $P_{v}(x^{\prime})$. Thus
\[
P_{v}(x^{\prime})\geq\frac{7}{8}\left(  \frac{1}{3}\Vert
x-x^{\prime} \Vert\right)  >\frac{1}{4}\Vert x-x^{\prime}\Vert.
\]
This yields
\begin{equation}
\mathcal{H}^{1}(P_{v}(\Omega(x)))+\frac{1}{4}\Vert x-x^{\prime}\Vert
\leq\mathcal{H}^{1}(P_{v}(\Omega(x^{\prime})))\;\hbox{ for $v \in
V$.} \label{es-2}
\end{equation}
Integrating \eqref{es-2} for $v\in V$ and \eqref{es-1} for $v\in
\mathbb{S}^{n-1}\setminus V$, and summing up the resulting inequalities we
obtain \eqref{estim2} with
\[
\varepsilon=(4\sigma_{n})^{-1}\int_{V}du.
\]
Notice that  this bound only depends on $\delta$, so the claim follows.\hfill
$\blacktriangle$

\bigskip

Consequently, our next goal is to determine $\delta>0$ so that the
assertion of Claim~2 holds. The exact value of the parameter
$\delta$ is eventually given in \eqref{valueofdelta} and depends
only on the dimension. In particular, it works for any
self-contracted curve and any choice of points
$x,x^{\prime}\in\Gamma$. For the remaining part of the proof, it is
possible to replace $\delta$ by its precise value. Nevertheless, we
prefer not to do so, in order to illustrate how this value is
obtained. In the sequel, the only \textit{prior} requirement is the
bound $\delta\leq2^{-4}$.

Fix any $x,x^{\prime}$ in $\Gamma$ with $x^{\prime}\prec x$ and recall the
definition of $x_{0},$ $v_{0}$ in \eqref{guy} and $\xi_{0}(y)$, for
$y\in\Gamma(x)$ in \eqref{guy-1}. Based on this, we consider the orthogonal
decomposition
\begin{equation}
\mathbb{R}^{n}\approx\mathbb{R}v_{0}\oplus\mathcal{U} \label{ad-1}
\end{equation}
where $\mathcal{U}=\left(  v_{0}\right)  ^{\perp}$ is the orthogonal
hyperplane to $v_{0}.$ The vector $\overline{v}$ of the assertion of
Claim~2 will be taken of the form
\begin{equation}
\overline{v}=\frac{v_{0}-\delta\zeta}{\Vert v_{0}-\delta\zeta\Vert},
\label{vvww}
\end{equation}
where $\zeta$ is a unit vector in $\mathcal{U}$. This vector will be
determined later on, as an application of Lemma~\ref{caratheodory}; notice
however that for any $\zeta\in\mathbb{S}^{n-1} \cap\mathcal{U}$ we get
$\overline{v}\in\mathbb{S}^{n-1}\cap B(v_{0},\delta)$, as needed.

We set
\begin{equation}
\Gamma_{0}=\big\{y\in\Gamma(x)\,;\,\langle
v_{0},\xi_{0}(y)\rangle\leq -2\delta\big\}. \label{gamma-0}
\end{equation}
Notice that for every $y\in\Gamma_{0}$ and any $\overline{v}$ of the form
\eqref{vvww} we have
\[
\langle\bar{v},\xi_{0}(y)\rangle\leq-2\delta+\|v_{0}-\bar{v}\|\leq-\delta
\leq-\delta^{2}.
\]
Thus \eqref{CL21} is satisfied for all $y$ in $\Gamma_{0}$.

It remains to choose $\zeta$ (and adjust the value of $\delta$) so that
\eqref{CL21} would also hold for $y\in\Gamma(x)\setminus\Gamma_{0}.$ This will
be done in five steps. In the sequel we shall make use of the decomposition
\eqref{ad-1} of vectors $\xi_{0}(y),$ $y\in\Gamma(x)$, namely:
\begin{equation}
\xi_{0}(y)=\langle
v_{0},\xi_{0}(y)\rangle\,v_{0}+\xi_{0}^{\mathcal{U} }(y),
\label{ad-2}
\end{equation}
where $\xi_{0}^{\mathcal{U}}(y)$ is the orthogonal projection of $\xi_{0}(y)$
in $\mathcal{U}$.

\medskip

\textit{Step 1}. We establish that for all $y\in\Gamma(x)\setminus\{x\}$ it
holds
\begin{equation}
\langle
v_{0},\xi_{0}(y)\rangle\leq-\frac{\|x-x^{\prime}\|}{6\|y-x_{0}\|}<0.
\label{ad-3}
\end{equation}
Indeed, since $x^{\prime}\prec x\prec y$, we deduce from self-contractedness
that $\Vert y-x^{\prime}\Vert\geq\Vert y-x\Vert$, that is, $y$ lies at the
same half-space as $x$ defined by the mediatrix hyperplane of the segment
$[x,x^{\prime}]$. Hence, denoting by $P_{v_{0}}$ the orthogonal projection of
$\mathbb{R}^{n}$ on ${\mathbb{R}}v_{0}$ (which we brutally identify to
${\mathbb{R}}$ to write inequalities) we observe that
\[
P_{v_{0}}(y)\leq\frac{1}{2}P_{v_{0}}(x)+\frac{1}{2}P_{v_{0}}(x^{\prime
})=P_{v_{0}}(x)+\frac{1}{2}\Vert x-x^{\prime}\Vert,
\]
and consequently
\[
\langle v_{0},y-x_{0}\rangle=P_{v_{0}}(y-x_{0})\leq
P_{v_{0}}(x)+\frac{1}{2}\Vert
x-x^{\prime}\Vert-P_{v_{0}}(x_{0})=-\frac{1}{6}\Vert x-x^{\prime
}\Vert,
\]
thus, dividing by $\|y-x_{0}\|$, \eqref{ad-3} follows.

\bigskip

\textit{Step 2}. We establish that for all $y\in\Gamma(x)\setminus\Gamma_{0}$,
we have:
\begin{equation}
\|y-x_{0}\|>\frac{1}{12\delta}\|x-x^{\prime}\|\,; \label{ad-4}
\end{equation}
\begin{equation}
|\langle v_{0},\xi_{0}(y)\rangle|=\|\xi_{0}(y)-\xi_{0}^{\mathcal{U}}
(y)\|\leq2\delta\,; \label{ad-45}
\end{equation}
and
\begin{equation}
\sqrt{1-4\delta^{2}}\leq\|\xi_{0}^{\mathcal{U}}(y)\|\leq1\,.
\label{ad-6}
\end{equation}
\bigskip

Indeed, since $y\in\Gamma(x)\setminus\Gamma_{0}$, \eqref{ad-4}
follows easily by combining \eqref{gamma-0} with \eqref{ad-3}. The
same formulas yield that $\langle
v_{0},\xi_{0}(y)\rangle\in(-2\delta,0),$ whence $|\langle
v_{0},\xi_{0}(y)\rangle|\leq2\delta.$ In view of \eqref{ad-2}, both
\eqref{ad-45} and \eqref{ad-6} follow directly. This ends the proof
of Step 2.

\bigskip

Before we proceed, let us make the following observation, which
motivates Step 4 (compare forthcoming inequalities \eqref{positive}
and \eqref{casi-pos}). We set
\begin{equation}
\xi(y)=\frac{y-x}{\Vert y-x\Vert}\in\mathbb{S}^{n-1},\quad\text{for all
}y\in\Gamma(x)\setminus \{x\}. \label{ksi}
\end{equation}
It is easily seen that
\begin{equation}
\langle\xi(y),\xi(z)\rangle\geq0\text{, \quad for all }y,z\in\Gamma
(x)\setminus \{x\}. \label{positive}
\end{equation}
Indeed assuming $x\preceq y\prec z$ (the case $y=z$ is trivial), the
definition of self-contractedness yields
\[
\Vert z-x\Vert\geq\Vert z-y\Vert.
\]
Writing $z-y=(z-x)-(y-x)$ and squaring the above inequality, we obtain
\[
\Vert z-x\Vert^{2}\geq\Vert z-x\Vert^{2}+\Vert x-y\Vert^{2}-2\langle
z-x,y-x\rangle,
\]
which yields $\langle z-x,y-x\rangle\geq0$. Dividing by norms gives \eqref{positive}.

\bigskip

Our next objective is to show that vectors $\xi_{0}^{\mathcal{U}}(y),$
$y\in\Gamma(x)\setminus\Gamma_{0}$ satisfy a relaxed inequality of type
\eqref{positive}. We need the following intermediate step.

\bigskip

\textit{Step 3}. We show that for every $y\in\Gamma(x)\setminus\Gamma_{0},$ we
have:
\begin{equation}
\|\xi(y)-\xi_{0}(y)\|\,\leq\,32\,\delta. \label{ad-30}
\end{equation}

Indeed
\begin{align*}
\|\xi(y)-\xi_{0}(y)\|\  &
\leq\Big\|\frac{y-x}{\|y-x\|}-\frac{y-x_{0}}
{\|y-x\|}\Big\|\ +\ \Big\|\frac{y-x_{0}}{\|y-x\|}-\frac{y-x_{0}}{\|y-x_{0}\|}\Big\|\\
&  =\frac{\|x-x_{0}\|}{\|y-x\|}\ +\
|\frac{1}{\|y-x\|}-\frac{1}{\|y-x_{0}
\|}|\,\|y-x_{0}\|\\
&  \leq\frac{2\|x-x_{0}\|}{\|y-x\|}=\frac{4}{3}\frac{\|x-x^{\prime}
\|}{\|y-x\|}.
\end{align*}
On the other hand, using \eqref{ad-4} we get:
\[
\|y-x\|\geq\|y-x_{0}\|-\|x_{0}-x\|\geq(\frac{1}{12\delta}-\frac{2}
{3})\|x-x^{\prime}\|=\left(  \frac{1-8\delta}{16\delta}\right)
\left( \frac{4\|x-x^{\prime}\|}{3}\right)  .
\]
Combining the above inequalities, we deduce
\[
\|\xi(y)-\xi_{0}(y)\|\ \leq\frac{16\delta}{1-8\delta},
\]
thus \eqref{ad-30} follows thanks to the prior bound $\delta\leq2^{-4}.$

\bigskip

\textit{Step 4}. We now establish that for every $y,z\in\Gamma(x)\setminus
\Gamma_{0},$ we have:
\begin{equation}
\langle\xi_{0}^{\mathcal{U}}(y),\xi_{0}^{\mathcal{U}}(z)\rangle\geq
-65\,\delta. \label{casi-pos}
\end{equation}

Let us first calculate the scalar product of the unit vectors
$\xi_{0}(y)$ and $\xi_{0}(z)$.

\begin{align*}
\langle\xi_{0}(y),\xi_{0}(z)\rangle &  =\langle\xi(y),\xi_{0}(z)\rangle
+\langle\xi_{0}(y)-\xi(y),\xi_{0}(z)\rangle\\
&  =\langle\xi(y),\xi(z)\rangle+\langle\xi(y),\xi_{0}(z)-\xi(z)\rangle
+\langle\xi_{0}(y)-\xi(y),\xi_{0}(z)\rangle\\
&  \geq0-\|\xi_{0}(z)-\xi(z)\|-\|\xi_{0}(y)-\xi(y)\|\geq-64\,\delta\ ,
\end{align*}
where \eqref{positive} and \eqref{ad-30} have been used for the last estimation.

Using now the above estimation, the orthogonal decomposition \eqref{ad-2}
recalling \eqref{ad-45} and the prior bound $\delta\leq2^{-4}$ we obtain
successively:
\begin{align*}
\langle\xi_{0}^{\mathcal{U}}(y),\xi_{0}^{\mathcal{U}}(z)\rangle &  =\langle
\xi_{0}(y),\xi_{0}(z)\rangle-\langle v_{0},\xi_{0}(y)\rangle\langle v_{0}
,\xi_{0}(z)\rangle\\
&  \geq-64\delta-4\delta^{2}\geq-65\delta.
\end{align*}

\textit{Step 5}. We shall now be interested in the set
\[
\Sigma=\big\{\hat{\xi}_{0}^{\mathcal{U}}(y)\,;\,y\in\Gamma(x)\setminus
\Gamma_{0}\big\}\subset\mathcal{U}\cap\mathbb{S}^{n-1},
\]
where $\hat{\xi}_{0}^{\mathcal{U}}(y)=\xi_{0}^{\mathcal{U}}(y)/\|\xi
_{0}^{\mathcal{U}}(y)\|$ (notice that $\xi_{0}^{\mathcal{U}}(y)\neq\emptyset$
from \eqref{ad-6}). In view of \eqref{casi-pos} and the estimation
\eqref{ad-6}, we obtain that for all $y,z\in\Gamma(x)\setminus\Gamma_{0}$
\[
\langle\hat{\xi}_{0}^{\mathcal{U}}(y),\hat{\xi}_{0}^{\mathcal{U}}
(z)\rangle\geq\frac{-65\delta}{1-4\delta^{2}}\geq-81\delta.
\]
We are now ready to give a precise value for $\delta$, namely,
\begin{equation}
\delta=\Big(\frac{1}{3}\Big)^{3n}. \label{valueofdelta}
\end{equation}
Recalling that $n\geq2,$ under this choice the scalar product between two
elements of $\Sigma$ is bounded from below by the quantity
\[
-81\,\delta\geq-81\,\Big(\frac{1}{3}\Big)^{3n}\,\geq-\Big(\frac{1}{3}\Big)^{n}.
\]
Therefore, the subset $\Sigma$ of the unit sphere of $\mathcal{U}
\simeq{\mathbb{R}}^{n-1}$ satisfies the assumptions of
Lemma~\ref{caratheodory}. We deduce that there exists a unit vector
$\zeta \in\mathcal{U}\cap\mathbb{S}^{n-1}$ such that
\begin{equation}
\langle\zeta,\hat{\xi}_{0}^{\mathcal{U}}(y)\rangle\geq\big(\frac{1}{3}
\big)^{2n-1},\text{ for all }y\in\Gamma(x)\setminus\Gamma_{0}.
\label{valueof zeta}
\end{equation}
Taking $\delta$ and $\zeta$ as above in \eqref{vvww}, we obtain $\overline
{v}\in\mathbb{S}^{n-1}\cap B(v_{0},\delta)$. Let us now verify that
\eqref{CL21} holds. As already mentioned (see comments before Step~1), it
remains to consider the case $y\in\Gamma(x)\setminus\Gamma_{0}.$ Notice that
in view of \eqref{ad-3}
\[
\langle v_{0}-\delta\zeta,\xi_{0}(y)\rangle=\langle v_{0},\xi_{0}
(y)\rangle-\delta\,\langle\zeta,\xi_{0}(y)\rangle\,\leq\,-\delta\,\langle
\zeta,\xi_{0}(y)\rangle,
\]
while in view of \eqref{ad-1}, \eqref{ad-6} and the orthogonality between
$v_{0}$ and $\zeta$ we get
\[
\langle\zeta,\xi_{0}(y)\rangle=\langle\zeta,\xi_{0}^{\mathcal{U}}
(y)\rangle=\|\xi_{0}^{\mathcal{U}}(y)\|\,\langle\zeta,\hat{\xi}_{0}
^{\mathcal{U}}(y)\rangle\geq\,\sqrt{1-4\delta^{2}}\,\big(\frac{1}{3}\big)^{2n-1}
\,\geq\big(\frac{1}{3}\big)^{2n}.
\]
Assembling the above, and using the fact that $\|v_{0}-\delta\zeta
\|=\sqrt{1+\delta^{2}}>1$ we obtain
\[
\langle\bar{v},\xi_{0}(y)\rangle\leq-\delta\sqrt{1+\delta^{2}}\Big(\frac{1}
{3}\Big)^{2n}\leq-\delta\Big(\frac{1}{3}\Big)^{3n}=-\delta^{2}.
\]
This shows that the assumption made in Claim~2 is always fulfilled.
The proof is complete.
\end{proof}

\subsection{Arc-length reparameterization}
\label{sec32} Having in hand that the length of a bounded
self-contracted curve $\gamma:I\rightarrow{\mathbb{R}}^{n}$ is
finite (\emph{c.f.} Theorem \ref{main1}), the curve $\gamma$ may be
reparameterized by its arc-length. This reparameterization is
particularly interesting when the curve $\gamma$ is
\emph{continuous.} In this case, we shall see that the arc-length
parametrization is Lipschitz continuous and that the image
$\Gamma=\gamma(I)$ of the curve can also be realized by a self-expanded
curve.

Set $I(t)=I\cap(-\infty,t]$, for $t\in I$, and consider the
length function
\[
\left\{
\begin{array}
[c]{l}
L:I\rightarrow\lbrack0,\ell(\gamma)]\smallskip\\
L(t)=\ell(\gamma_{|I(t)}).
\end{array}
\right.
\]
Thus $L(t)$ is the length of the truncated curve $\gamma|_{I(t)}$
and consequently the function $t\mapsto L(t)$ is nondecreasing, but
not necessarily injective ---it is locally constant whenever
$t\mapsto\gamma(t)$ is locally stationary, see also
Remark~\ref{scmetric} (i). Let us assume that $\gamma$ is bounded and continuous.
Then, in view of Remark~\ref{scmetric}\ (ii), the curve can be
extended continuously to $\bar{I}$ and the set $\bar{\Gamma
}=\gamma(I)\cup\{\gamma_{\infty}\}=\gamma(\bar{I})$ is a compact
connected arc. Moreover, $t\mapsto L(t)$ is continuous and
$L(\bar{I})=[0,\ell (\gamma)].$ It follows easily that $L^{-1}(s)$
is a (possibly trivial) interval $[t_{1},t_{2}]$. Consequently, the
function
\begin{equation}
\left\{
\begin{array}
[c]{l}
\tilde{\gamma}:[0,\ell(\gamma)]\rightarrow{\mathbb{R}}^{n}\smallskip\\
\tilde{\gamma}(s)=\gamma(t),\text{ for any }t\in L^{-1}(s)
\end{array}
\right.  \label{twain}
\end{equation}
is well-defined.

\begin{proposition}
[Hidden regularity of continuous self-contracted curves]\label{Prop_hidden}Let
$\gamma:I\rightarrow{\mathbb{R}}^{n}$ be a bounded continuous self-contracted
curve. Then \eqref{twain} defines a Lipschitz self-contracted curve with the
same image $\Gamma=\gamma(I)=\tilde{\gamma}([0,\ell(\gamma))).$ In particular,
$\Gamma$ is also the image of some self-expanded curve.
\end{proposition}

\begin{proof} Let $s,s^{\prime}\in\lbrack0,\ell(\gamma)]$
be given. Let $t,t^{\prime}\in I$ be such that $s=L(t)$ and
$s^{\prime}=L(t^{\prime})$. Assume that $s<s^{\prime}$ (thus a fortiori
$t<t^{\prime}$). We claim that
\begin{equation}
L(t^{\prime})\geq L(t)+d(\gamma(t),\gamma(t^{\prime})).\label{lip1}
\end{equation}
Indeed, recalling \eqref{length-formula}, for any $\varepsilon>0$,
there exists a finite sequence $t_{0}<\ldots<t_{m}$ in $I(t)$,
satisfying
\[
\Sigma_{1}=:\sum_{i=0}^{m-1}d(\gamma(t_{i}),\gamma(t_{i+1}))\geq
L(t)-\varepsilon.
\]
Since adding an extra point can only make the sum larger, we may
assume that $t_{m}=t$. Then add the point $t_{m+1}=t^{\prime}$ and
notice that $t_{0}<\ldots<t_{m}<t_{m+1}$ is a sequence in
$I(t^{\prime})$, thus
\[
\Sigma_{2}=:\sum_{i=0}^{m}d(\gamma(t_{i}),\gamma(t_{i+1}))=\Sigma_{1}
+d(\gamma(t),\gamma(t^{\prime}))\leq L(t^{\prime}).
\]
We deduce that
\[
L(t)-\varepsilon\leq\Sigma_{1}=\Sigma_{2}-d(\gamma(t),\gamma(t^{\prime}))\leq
L(t^{\prime})-d(\gamma(t),\gamma(t^{\prime})).
\]
Taking the limit as $\varepsilon\rightarrow0$ and get \eqref{lip1}.
It follows that
\[
\Vert\tilde{\gamma}(s^{\prime})-\tilde{\gamma}(s)\Vert=\Vert\gamma(t^{\prime
})-\gamma(t)\Vert\leq L(t^{\prime})-L(t)=s^{\prime}-s,
\]
which proves that $\tilde{\gamma}$ is $1$-Lipschitz on $[0,\ell(\gamma
)].$

It follows easily that $\tilde{\gamma}$ is self-contracted, while
the fact that $\tilde{\gamma}([0,\ell(\gamma)))=\gamma(I)$ is
straightforward. The last assertion follows from
Lemma~\ref{reverse}.
\end{proof}

\begin{remark}
[Improving the length bound]\label{Rem_improve}Although the
constants $C>0$ in  \eqref{mp-1} (for self-expanded curves) and
in \eqref{length} (for bounded self-contracted curves) stem from the
same method, the latter is based on nonsmooth arguments encompassing
nonregular curves and therefore it is not optimized. As a
consequence of Proposition~\ref{Prop_hidden}, given a continuous
self-contracted curve we can use \eqref{mp-1} to improve \emph{a
posteriori} the constant that bounds the length given
in~\eqref{length}. Notwithstanding, this trick cannot be applied to
any bounded self-contracted curve: in the discontinuous case, one
can still consider the length mapping $L(t)$, and use it to
construct a larger curve (whose image contains $\Gamma$) with a
Lipschitz parametrization, but which is possibly not
self-contracted.
\end{remark}

\bigskip

\section{Applications}

\subsection{Orbits of convex foliations}

\label{Section_foliations} As mentioned in the introduction,
self-contracted curves are naturally defined in a metric space
without prior regularity assumptions. Nonetheless, this notion has
been conceived to capture the behaviour of orbits of gradient
dynamical systems for (smooth) quasiconvex potentials in Euclidean
spaces~\eqref{diffeq}, or more generally, orbits of the subgradient
semi-flow defined by a nonsmooth convex function, see
\eqref{subgrad}. In both cases, orbits are at least, absolutely
continuous. In this section we broaden the above framework by
introducing a new concept of generalized solution to a nonsmooth
convex foliation, admitting (merely) continuous curves as
generalized orbits.

\subsubsection{Nonsmooth convex foliations}

Let us first give the definition of a convex foliation.

\begin{definition}
[Convex foliation]\label{Definition_foliation} A collection
$\{C_{r}\}_{r\geq0}$ of nonempty convex compact subsets of
$\mathbb{R}^{n}$ is a (global) \emph{convex foliation} of
$\mathbb{R}^{n}$ if
\begin{equation}
r_{1}<r_{2}\Longrightarrow C_{r_{1}}\subset\mathrm{int}\,C_{r_{2}}
\label{f11}
\end{equation}
and
\begin{equation}
\bigcup\nolimits_{r\geq0}\partial
C_{r}=\mathbb{R}^{n}\diagdown\mathrm{int} \,C_{0}. \label{f12}
\end{equation}

\end{definition}

In view of \eqref{f11}, relation \eqref{f12} is equivalent to the
fact that for every
$x\in\mathbb{R}^{n}\setminus\mathrm{int}\,C_{0},$ there exists a
\emph{unique} $r\in\lbrack0,\infty)$ such that $x\in\partial C_{r}.$
The above definition is thus equivalent to
\cite[Definition~(6.5)]{DLS}.

\begin{example}
[Foliation given by a function]

\textrm{(i)} The sublevel sets of a proper
(coercive) convex function $f:\mathbb{R}^{n}\rightarrow\mathbb{R}$ provide a
typical example of a convex foliation. Indeed, if $m=\min f$, we set
$C_{0}=f^{-1}(m)$ and $C_{r}=[f\leq m+r]$. Then \eqref{f12} follows from the fact that $f$ cannot be locally constant outside of $C_0$.

\textrm{(ii)} The
sublevel sets of a (coercive) quasiconvex function might fail to satisfy
\eqref{f12}, unless  $\mathrm{int}\,[f\leq\lambda]=[f<\lambda]$ for
all $\lambda>\min f.$ This condition is automatically satisfied if the
quasiconvex function is smooth and every critical point is a global minimizer.
A more general condition is to assume that $f$ is \emph{semistrictly
quasiconvex}, see for instance \cite{mY} for the relevant definition and characterizations.
\end{example}

\subsubsection{Generalized solutions and self-contractedness}

We consider a possibly unbounded interval $I=[0,T_{\infty})$ of
$\mathbb{R}$, where $T_{\infty}\in\mathbb{R}_{+}\cup\{+\infty\}$ and
a continuous injective curve $\gamma:I\rightarrow\mathbb{R}^{n}$.
For every $\tau\in I$ we define the set of all possible limits of
\emph{backward secants} at $\gamma(\tau)$ as follows:
\begin{equation}
\mathrm{\sec}^{-}(\tau):=\left\{  q\in\mathbb{S}^{n-1}:q=\lim_{t_{k}
\nearrow\tau^{-}}\,\frac{\gamma(t_{k})-\gamma(\tau)}{\|\gamma(t_{k})-\gamma(\tau)\|}\right\}
,\label{n11}
\end{equation}
where the notation $\{t_{k}\}_{k}\nearrow\tau^{-}$ indicates that
$\{t_{k}\}_{k}\rightarrow\tau$ and $t_{k}<\tau$ for all $k$. The set
$\mathrm{\sec}^{+}(\tau)$ is defined analogously using decreasing sequences
$\{t_{k}\}_{k}\searrow\tau^{+}$. The compactness of $\mathbb{S}^{n-1}$ guarantees
that both $\mathrm{\sec}^{-}(\tau)$ and $\mathrm{\sec}^{+}(\tau)$ are
nonempty. The following lemma relates $\mathrm{\sec}^{-}(\tau)$ and
$\mathrm{\sec}^{+}(\tau)$ with the derivative of $\gamma,$ in case the latter
exists and does not vanish.

\begin{lemma}
[Secants versus derivative]\label{Lemma_sec1}Assume
$\gamma:I\rightarrow \mathbb{R}^{n}$ is differentiable at $\tau\in
I$ and $\gamma'(\tau )\neq0.$ Then
\[
\mathrm{\sec}^{-}(\tau)=\left\{  -\frac{\gamma'(\tau)}{\|\gamma'(\tau)\|}\right\}  \text{\quad and\quad}\mathrm{\sec}^{+}(\tau)=\left\{
\frac{\gamma'(\tau)}{\|\gamma'(\tau)\|}\right\}  .
\]

\end{lemma}

\begin{proof} Let $q\in\mathrm{\sec}^{-}(\tau)$ and let
$\{t_{k}\}_{k}\nearrow\tau^{-}$ be a sequence realizing this limit. Then
writing
\[
\frac{\gamma(t_{k})-\gamma(\tau)}{\|\gamma(t_{k})-\gamma(\tau)\|}=\left(
\frac{\gamma(t_{k})-\gamma(\tau)}{t_{k}-\tau}\right)  \left(  \frac
{-|\tau-t_{k}|}{\|\gamma(t_{k})-\gamma(\tau)\|}\right)
\]
and passing to the limit as $\{t_{k}\}_{k}\nearrow\tau^{-}$ we obtain the
result. The second assertion follows similarly.
\end{proof}

Assuming $\gamma(t)\notin\mathrm{int}\,C_{0}$ we denote by $r(t)$ the unique
positive number such that
\begin{equation}
\gamma(t)\in\partial C_{r(t)}. \label{n12}
\end{equation}
We are now ready to give the definition of generalized orbit to a convex foliation.

\begin{definition}
[Convex foliation orbits]\label{Definition_orbit-foliation} A
continuous curve $\gamma:I\rightarrow\mathbb{R}^{n}$ is called
generalized solution (orbit) of the convex foliation
$\{C_{r}\}_{r\geq0}$ if $\gamma(I)\cap\mathrm{int}
\,C_{0}=\emptyset$ and under the notation \eqref{n11} and \eqref{n12}
we have:

\begin{itemize}
\item[(i)] the function $t\mapsto r(t)$ is decreasing ;

\item[(ii)] $\mathrm{\sec}^{-}(t)\subset N_{C_{r(t)}}(\gamma(t))$ for all
$t\in I.$
\end{itemize}
\end{definition}

It follows directly from (i) that $\gamma$ is injective, and then (ii) makes sense. Also, every orbit of a convex
foliation is an injective mapping with bounded image ($\Gamma=\gamma(I)$ is
contained in the compact set $C_{r(0)}$). Recall the definition of the normal cone $N_{C_{r(t)}}$ in \eqref{NKu0}. If $\partial C_{r(t)}$ is a smooth
manifold (or more generally, if $\dim N_{C_{r(t)}}(\gamma(t))=1$) we have
$\mathrm{\sec}^{-}(t)=\{\nu_{r(t)}\}$ where $\nu_{r(t)}$ is the unit normal of
$\partial C_{r(t)}$ at $\gamma(t).$

In the sequel, we shall need the following lemma, which proof is an
easy exercise and will be omitted.

\begin{lemma}
[Criterium for decrease]\label{Lem_trivial}Let
$\varphi:[0,T]\rightarrow \mathbb{R}$ be a continuous function,
where $T\geq0$. Assume that for every $0<\tau\leq T,$ there exists
$\delta>0$ such that for all $t\in(\tau
-\delta,\tau)\cap\lbrack0,T]$ we have
\begin{equation}
\varphi(t)>\varphi(\tau).\label{dan-1}
\end{equation}
Then $\varphi$ is decreasing.
\end{lemma}

We now prove that the generalized solutions of a convex foliation
are self-contracted curves. Notice that in view of
Theorem~\ref{main1} this entails that these curves are rectifiable
and have a finite length.

\begin{theorem}
[Self-contractedness of generalized orbits]\label{Theorem_orbit-foliation}
Every convex foliation orbit is a (continuous) self-contracted curve.
\end{theorem}

\begin{proof} Let $\gamma:I\rightarrow\mathbb{R}^{n}$ be
a continuous injective curve which is a generalized solution of the
convex foliation $\{C_{r}\}_{r\geq0}$ in the sense of
Definition~\ref{Definition_orbit-foliation}. Since $\gamma(I)\subset
C_{r(0)}$, the curve is obviously bounded. Fix $\tau\in I$ and
$\varepsilon>0$ such that $\tau+\varepsilon\in I$. We set
\[
\Gamma_{\varepsilon}:=\{\gamma(t):t\geq\tau+\varepsilon,\ t\in I\}.
\]
Notice that the compact set $\overline{\Gamma_{\varepsilon}}$ is contained in
$\mathrm{int}\,C_{r(\tau+\varepsilon^{\prime})}$ for all $0\leq\varepsilon
^{\prime}<\varepsilon.$

\medskip

\textit{Claim~1}. There exists $a>0$, such that for all $v\in
N_{C_{r(\tau)} }(\gamma(\tau))$ and all
$x\in\overline{\Gamma_{\varepsilon}}$ we have
\begin{equation}
\langle\,v,\,\frac{x-\gamma(\tau)}{\|x-\gamma(\tau)\|}\,\rangle
\,<\,-a\,<\,0\,. \label{existsa}
\end{equation}

\textit{Proof of Claim~1}. Assume, towards a contradiction, that
\eqref{existsa} fails. Since the sets $\overline{\Gamma_{\epsilon}}$
and $N_{C_{r(\tau)}}(\gamma(\tau))\cap\mathbb{S}^{n-1}$ are compact,
we easily deduce that for some unit normal $v\in
N_{C_{r(\tau)}}(\gamma(\tau))$ and
$x\in\overline{\Gamma_{\epsilon}}$ we have $\langle
v,x-\gamma(\tau)\rangle \geq 0$. Since $x\in\mathrm{int\,}C_{r(\tau)}$,
we deduce that for some $y\in C_{r(\tau)}$ we have $\langle
v,y-\gamma(\tau)\rangle>0$. This contradicts the fact that $v\in
N_{C_{r(\tau)}}(\gamma(\tau))$. Thus \eqref{existsa}
holds.\hfill$\blacktriangle$

\textit{Claim~2}. For every $\varepsilon>0$ there exists $\delta>0$
such that for all $t^{-},t^{+}\in I$ such that
$t^{-}\in\lbrack\tau-\delta,\tau)$ and $t^{+}\geq\tau+\varepsilon$
we have
\begin{equation}
\langle\gamma(t^{-})-\gamma(\tau),\,\gamma(t^{+})-\gamma(\tau)\rangle
\,<0\,.\label{inequaa}
\end{equation}

\textit{Proof of Claim~2}. Fix $\varepsilon>0$ and let $a>0$ be
given by Claim~1. Let
$N^{1}=N_{C_{r(\tau)}}(\gamma(\tau))\cap\mathbb{S}^{n-1}$ denote the (compact)
set of unit normals at $\gamma(\tau)$, let
$U_{\alpha}$:$=N^{1}+B(0,\alpha)$ be its $\alpha$-enlargement, and
let $N_{\alpha}$ denote the closed convex cone generated by
$U_{\alpha}$. It follows from \eqref{existsa} that if $\alpha > 0$ is chosen small enough (depending on $a$), then for all $w\in
N_{\alpha}$, and all $x\in\overline{\Gamma_{\epsilon}}$, we have
$\langle w,x-\gamma(\tau)\rangle<0$. Since
$\mathrm{\sec}^{-}(\tau)\subset N^{1}\subset U_{a}$, we deduce that
for some $\delta>0$ we have
\[
\frac{\gamma(t)-\gamma(\tau)}{\Vert\gamma(t)-\gamma(\tau)\Vert}\in U_{\alpha
},\quad\text{for all }t\in(\tau-\delta,\tau)\cap I.
\]
This establishes \eqref{inequaa}.\hfill$\blacktriangle$

\medskip

Notice that \eqref{inequaa} yields
\begin{equation}
\|\gamma(t^{-})-\gamma(t^{+})\|\,>\,\|\gamma(\tau)-\gamma(t^{+})\|.\label{t11}
\end{equation}
We now prove that the curve $\gamma$ is self-contracted. Indeed, fix
$0\leq t_{1}<t_{2}<t_{3}$ in $I$ and consider the real-valued
function
\[
\left\{
\begin{array}
[c]{l}
\varphi:[0,t_{2}]\rightarrow\mathbb{R}\medskip\\
\varphi(t)=\|\gamma(t)-\gamma(t_{3})\|
\end{array}
\right.
\]
Applying Claim~2 for any $0<\tau\leq t_{2}$ and for
$\varepsilon=t_{3}-t_{2}$ we deduce that for some $\delta_{1}>0$ and
all $t\in(\tau-\delta_{1},\tau )\cap\lbrack0,t_{2}]$ we have
$\varphi(t)>\varphi(\tau).$ The conclusion follows from
Lemma~\ref{Lem_trivial}. Since $t_{1},t_{2},t_{3}$ are arbitrarily
chosen we deduce that $\gamma$ is self-contracted.
\end{proof}

As a consequence of Theorem~\ref{Theorem_orbit-foliation} and
\cite[Proposition~2.2]{DLS} the solution curve $\gamma$ has finite
length. In particular, the curve converges as $t\rightarrow
T_{\infty}$ to some limit point
\[
\gamma_{\infty}:=\lim_{t\rightarrow T_{\infty}}\gamma(t)\in\partial
C_{r(T_{\infty})}
\]
so that the mapping $\gamma$ can be continuously extended to
$[0,T_{\infty}]$, by setting $\gamma(T_{\infty}):=\gamma_{\infty}$.

\begin{remark}
[Generalized versus classical solutions]\label{Rem-arisd} Let
$\gamma :I\rightarrow\mathbb{R}^{n}$ be a convex foliation orbit and
$\Gamma =\gamma(I)$. Although $\gamma$ is merely continuous,
Theorem~\ref{Theorem_orbit-foliation} guarantees its rectifiability.
Thus Proposition~\ref{Prop_hidden} applies, and the curve can be
reparameterized by its arc-length to obtain a Lipschitz curve with
the same image. The new curve
$\gamma_{\ast}:[0,\ell(\gamma)]\rightarrow\mathbb{R}^n$ satisfies
$\|\gamma'_{\ast}(s)\|=1$ and
\[
-\gamma'_{\ast}(s)\in N_{C(r(s))}(\gamma_{\ast}(s))
\]
for almost all $s\in\lbrack0,\gamma(\ell)]$. It follows that every generalized
orbit of a convex foliation gives rise to a classical solution (Lipschitz orbit).
\end{remark}

\subsubsection{Smooth solutions and strong self-contractedness}

Let $\gamma:I\rightarrow\mathbb{R}^{n}$ be a continuous curve, where
$I=[0,T_{\infty})$, let $T\in I$ and denote by $\Omega(T)$ the
convex hull of the tail of the curve, that is,
\[
\Omega(T):=\mathrm{conv\,}\{\gamma(t):t\geq T\}.
\]
Let us give the following definitions.

\begin{definition}
[Strong self-contractedness]\label{def-strong-sc} A continuous injective curve
$\gamma:I\rightarrow \mathbb{R}^{n}$ is called

\begin{itemize}

\item \emph{strictly self-contracted,} if for $t_{1},t_{2},t_{3}$ in $I$ such
that $t_{1}<t_{2}<t_{3}$ we have
\[
\|\gamma(t_{1})-\gamma(t_{3})\|>\|\gamma(t_{2})-\gamma(t_{3})\|.
\]

\item \emph{strongly self-contracted,} if for all $\tau\in I$ we have
\begin{equation}
\mathrm{\sec}^{-}(\tau)\subset\mathrm{int\,}N_{\Omega(\tau)}(\gamma(\tau)).
\label{str-SC}
\end{equation}

\end{itemize}
\end{definition}

Of course every strictly self-contracted
curve is self-contracted. A careful look at the  proof of
Theorem~\ref{Theorem_orbit-foliation} actually reveals that the
orbits of the convex foliation are strictly self-contracted curves.
 The following proposition relates
the notions of strict and strong self-contractedness.

\begin{proposition}[Strong versus strict self-contractedness]
\label{Prop_strong-strict}Every strongly self-contracted curve is
strictly self-contracted.
\end{proposition}

\begin{proof} Fix $t_{1}<t_{2}<t_{3}$ in $I$ and apply
\eqref{str-SC} for $\tau=t_{2}$. From the definition of
$\mathrm{\sec} ^{-}(\tau)$, see \eqref{n11}, we deduce that for some
$\delta>0$ and all $t\in(t_{2}-\delta,t_{2})$ we have
\[
\langle\gamma(t)-\gamma(t_{2}),\,\gamma(t_{3})-\gamma(t_{2})\rangle<0.
\]
Setting $\varphi(t)=\|\gamma(t)-\gamma(t_{3})\|,$ the above
inequality guarantees that Lemma~\ref{Lem_trivial} applies. The
proof is complete.
\end{proof}

Our next objective is to show that smooth orbits of a convex foliation are
actually strongly self-contracted curves. We shall need the following lemma.

\begin{lemma}
[Differentiability point of a convex foliation
orbit]\label{Lemma_ss2}Let $\gamma:I\rightarrow\mathbb{R}^{n}$ be a
continuous orbit of a convex foliation and assume that $\gamma$ is
differentiable at some $\tau>0$ with $\gamma'(\tau)\neq0$. Then
\begin{equation}
-\gamma'(\tau)\in\mathrm{int\,}N_{\Omega(\tau)}(\gamma(\tau
)).\label{str-sc}
\end{equation}

\end{lemma}

\begin{proof} In order to simplify notation we assume $\gamma
(\tau)=0$. (There is no loss of generality in doing this.) In view of Lemma
\ref{Lemma_sec1} there exist $\varepsilon>0,$ such that for all $t^{+}\in
(\tau,\tau+\varepsilon)$ we have
\begin{equation}
\langle\frac{\gamma'(\tau)}{\|\gamma'(\tau)\|},\frac{\gamma(t^{+}
)}{\|\gamma(t^{+})\|}\rangle\,>\frac{1}{2}. \label{ad1}
\end{equation}
Since the compact set $\Omega(\tau+\varepsilon)$ is contained in
$\mathrm{int}\,C_{r(\tau)},$ arguing as in Claim~1 of the proof of
Theorem~\ref{Theorem_orbit-foliation}, we deduce that for some $a>0$
and for all $t^{+}\geq\tau+\varepsilon$ we have
\begin{equation}
\langle-\frac{\gamma'(\tau)}{\|\gamma'(\tau)\|},\frac{\gamma(t^{+}
)}{\|\gamma(t^{+})\|}\rangle<-a. \label{ad2}
\end{equation}
Changing $a$ into $\min(a,1/2)$ if necessary, we deduce that
\eqref{ad2} holds true for all $t^{+} > \tau$. We deduce that the
tail of the curve for $t\geq \tau$ is contained in the convex cone
$$\Big\{ x\in\mathbb{R}^n  \; : \; \langle \frac{\gamma'(\tau)}{\|\gamma'(\tau)\|} , x \rangle \leq -a \|x\|\Big\},$$
 hence so does $\Omega(\tau)$. This easily yields
$-\gamma'(\tau)\in\mathrm{int\,}N_{\Omega(\tau)}(\gamma
(\tau))$.
\end{proof}

\begin{corollary}
[Strong self-contractedness of smooth orbits]\label{Cor_strongSC}Every $C^{1}$
convex foliation orbit with no stationary point is a strongly self-contracted curve.
\end{corollary}

\begin{proof} Parameterizing the curve by its arc-length
parametrization we obtain a $C^{1}$ curve with nonvanishing
derivative and the same image $\Gamma=\gamma(I)$. This new curve
satisfies $\|\dot\gamma (s)\|=1$, for every $s\in
[0,\ell(\gamma)]$ and it is also a convex foliation orbit, see
Remark~\ref{Rem-arisd}. The result follows by combining
Lemma~\ref{Lemma_sec1} with Lemma \ref{Lemma_ss2}.
\end{proof}

\subsection{Polygonal approximations of smooth strongly self-contracted
curves}

\label{sec42} In this section we prove that every strongly
self-contracted $C^{1}$ curve is a limit of self-contracted
polygonal curves (with respect to the Hausdorff distance). Let us
recall the relevant definition.

\begin{definition}
[Polygonal approximation]Let $\gamma:I\rightarrow\mathbb{R}^{n}$ be
a continuous curve. A polygonal line
$P=\bigcup\limits_{k=0}^{m}[z_{k+1},z_{k}]$ is called polygonal
approximation of accuracy $\delta>0$ for the curve $\gamma$, if
\[
\{z_{k}\}_{k}\subset\gamma(I), \quad z_{k+1} \preceq z_{k}   \quad \text{ and } \quad d_{H}(P,\gamma(I))\leq
\delta.
\]
\end{definition}

In the above  we used the same notation $x\preceq y$ which corresponds to the order on the curve as in the beginning of the proof of Theorem \ref{main1}. In the sequel, bounded continuous self-contracted curves $\gamma
:I\rightarrow\mathbb{R}^{n}$ are considered to be extended to $\bar{I}$, so
that $\Gamma=\gamma(\bar{I})$ is a compact set.

\begin{proposition}
[Self-contracted polygonal approximations]\label{Prop_polyg}Let
$\gamma :[0,L]\rightarrow\mathbb{R}^{n}$ be a $C^{1}$ strongly
self-contracted curve. Then for every $\delta>0$ the curve $\gamma$
admits a self-contracted poly\-go\-nal approximation of accuracy
$\delta>0$.
\end{proposition}

\begin{proof} Since the statement is independent of the
parametrization, there is no loss of generality in assuming that the
curve is parameterized by arc-length. In
particular, $\|{\gamma}'(t)\|=1$ for all $t\in\lbrack0,L]$ where
$L$ is the total length of the curve. Moreover,
\eqref{str-sc} holds.

We shall show that for every $\delta>0$ and every $T_{0}\in(0,L]$
there exists a polygonal approximation
$P_{\delta}(T_{0})=\bigcup\limits_{k=0}^{m} [z_{k+1},z_{k}]$ of the
curve $\gamma([T_{0},L])$ of accuracy $\delta>0.$ Indeed, let
$T_{0}\in(0,L]$ and set $z_{0}=\gamma(T_{0}).$ In view of
\eqref{str-sc} and Lemma~\ref{Lemma_sec1} there exists
$t^{-}\in\lbrack T_{0}-\delta,T_{0})$ such that
\begin{equation}
\frac{\gamma(t^{-})-\gamma(T_{0})}{\|\gamma(t^{-})-\gamma(T_{0})\|}
\in\mathrm{int\,}N_{\Omega(T_0)}(\gamma(T_{0})).\label{ad3}
\end{equation}
Let $T_{1}>0$ be the infimum of all $t^{-}\in\lbrack T_{0}-\delta,T_{0})$ such
that \eqref{ad3} holds, and set $z_{1}=\gamma(T_{1}).$ Notice that
$(z_{1}-z_{0})\in N_{\Omega(T_0)}(\gamma(T_{0})),$ which yields that for any
$x\in\Omega(T_{0}),$ the function
\[
t\mapsto d(x,z_{0}+t(z_{1}-z_{0}))
\]
is increasing. Thus the curve $[z_{1},z_{0}]\cup\{\gamma(t):t\geq T_{0}\}$ is
self-contracted. Since $t\mapsto\gamma(t)$ is the length parametrization, it
follows that $\|z_{1}-z_{0}\|\leq|T_{1}-T_{0}|\leq\delta$ whence
\[
d_{H}([z_{1},z_{0}],\gamma([T_{1},T_{0}]))\leq\delta.
\]
Notice finally that
\[
\lbrack z_{1},z_{0}]\cup\Omega(T_{0})\subset\Omega(T_{1}).
\]
Repeating the above procedure, we build after $m$ iterations, a
 decreasing sequence $T_{0}>T_{1}>\ldots>T_{m+1}$ and points
$z_{k} =\gamma(T_{k}),$ such that
$\bigcup\limits_{k=0}^{m}[z_{k+1},z_{k}]$ is a self-contracted
polygonal approximation of the curve $\{\gamma(t):t\in\lbrack
T_{m+1},T_{0}]\}$ of precision $\delta>0$. Notice also that $\Omega(T_0) \cup \bigcup\limits_{k=0}^{m}[z_{k+1},z_{k}] \subset \Omega(T_{m+1})$.

The proof will be complete if we show that with this procedure we
reach the initial point $t=0$ in a finite number of iterations,
\emph{i.e. }$T_{m+1}=0$ for some $m\geq1$. Let us assume this is not
the case. Then the aforementioned procedure gives a
decreasing sequence $\{T_{m}\}_{m}\searrow T$, where $T\ge 0$. Set
$z=\gamma(T)$ and $z_{m}=\gamma(T_{m})$ and notice that  $\{z_{m}
\}\rightarrow z.$ We may assume that $\{z_{m}\}_{m\geq m_{0}}\subset
B(z,\delta)$. It follows that for every $m\geq m_{0}$ there exists
$s_{m}>0$ and $x_{m}:=\gamma(T_{m}+s_{m})$ such that
\begin{equation}
\langle\frac{z-z_{m}}{\|z-z_{m}\|},\frac{x_{m}-z_{m}}{\|x_{m}-z_{m}\|}
\rangle>0. \label{ad4}
\end{equation}
(If \eqref{ad4} were not true, then we could have taken $z_{m+1}=z$ and
$T_{m+1}=T$ a contradiction to the definition of $T_{m+1}$ as an infimum.) Let
us now assume that a subsequence of $\{x_{m}\}_{m}$ remains away of $z$. Then
taking a converging sub-subsequence $\{z_{k_{m}}\}_{m}$ and passing to the
limit in \eqref{ad4} as $m\rightarrow\infty,$ we obtain that $\langle
\gamma'(T),u\rangle\geq0$, for some unit vector $u$ in  the cone over $\Omega(T)-z$,
which contradicts \eqref{str-SC}. It follows that $s_{m}\rightarrow0$ and
$\{x_{m}\}_{m}\rightarrow z$.

For the rest of the proof, let us assume for simplicity that $z=0$ and
$\gamma'(T)=e_{n}=(0,\ldots,1)$. (There is no loss of generality in doing
so.) Since $\gamma$ is a $C^{1}$ curve, it follows that for any $\varepsilon
>0$, there exists $\tau>0$ such that for all $t\in(T,T+\tau)$ we have
\begin{equation}
1\geq\langle
e_{n},\gamma'(t)\rangle\geq1-\varepsilon\label{ad5}
\end{equation}
and
\begin{equation}
\Big\| e_{n}-\frac{\gamma(t)}{\|\gamma(t)\|} \Big\|\leq \varepsilon.
\label{ad6}
\end{equation}
Pick any $\bar{m}$ sufficiently large so that $T_{\bar{m}}+s_{\bar{m}}<\tau,$
set $\bar{z}=z_{\bar{m}}=\gamma(T_{\bar{m}})$ and $\bar{x}=x_{\bar{m}}
=\gamma(T_{\bar{m}}+s_{\bar{m}}).$ Since $\bar{x},\bar{z}$ satisfy
\eqref{ad4} and \eqref{ad6} we deduce that
\[
\langle\frac{\bar{z}}{\|\bar{z}\|},\frac{\bar{x}-\bar{z}}{\|\bar{x}-\bar{z}
\|}\rangle<0\text{\quad and\quad}\|e_{n}-\frac{\bar{z}}{\|\bar{z}
\|}\|<\varepsilon,
\]
hence $\langle e_n, \frac{\bar{x}-\bar{z}}{\|\bar{x}-\bar{z}
\|} \rangle < \varepsilon$, and the last coordinates satisfy
\begin{equation}
\bar{x}^{n}-\bar{z}^{n}<\varepsilon\|\bar{x}-\bar{z}\|\,=\varepsilon
\,\|\gamma(T_{\bar{m}}+s_{\bar{m}})-\gamma(T_{\bar{m}})\|\,\leq\,\varepsilon
s_{\bar{m}}. \label{ad7}
\end{equation}
On the other hand, using \eqref{ad5} we deduce
\begin{equation}
\bar{x}^{n}-\bar{z}^{n}=\int_{0}^{s_{\bar{m}}}\langle e_{n},gamma'
(T_{\bar{m}}+s)\rangle ds\geq(1-\varepsilon)s_{\bar{m} }.
\label{ad8}
\end{equation}
Relations \eqref{ad7}, \eqref{ad8} are incompatible for $\varepsilon<1/2,$
leading to a contradiction.

This proves that for some $m$ we get $T_{m+1}=0$ and $z_{m+1}=\gamma(0)$. The
proof is complete.
\end{proof}

In view of Corollary~\ref{Cor_strongSC} we thus obtain.

\begin{corollary}
Every $C^{1}$ convex foliation orbit admits self-contracted polygonal
approximations of arbitrary accuracy.
\end{corollary}

\subsection{Convergence of the proximal algorithm}

\label{Section_proximal}The aim of this section is to provide a
diffe\-rent type of application of the notion of self-contracted
curves. We recall that for a (nonsmooth) convex function
$f:\mathbb{R}^{n}\rightarrow\mathbb{R}$ which is bounded from below,
an initial point $x_{0}\in\mathbb{R}^{n}$ and a sequence
$\{t_{i}\}_{i}\subset(0,1]$, the algorithm defines a sequence
$\{x_{i}\}_{i\geq 0}\subset\mathbb{R}^{n}$ called \em proximal
sequence \em according to the iteration scheme, see
\cite[Definition~1.22]{Rock98} for example:

Given $x_{i}$, define $x_{i+1}$ as the (unique) solution of the following
(strongly convex, coercive) minimization problem
\begin{equation}
\min_{x}\,\left\{  f(x)+\frac{1}{2t_{i}}\|x-x_{i}\|^{2}\right\}
.\label{1r}
\end{equation}
A necessary and sufficient optimality condition for the above problem is
\[
0\in\partial f(x_{i+1})+t_{i}^{-1}(x_{i+1}-x_{i})
\]
or equivalently,
\begin{equation}
\frac{x_{i+1}-x_{i}}{t_{i}}\in -\partial f(x_{i+1}),\label{2r}
\end{equation}
where $\partial f$ is the Fenchel subdifferential of $f$. Notice
that \eqref{2r} can be seen as an implicit discretization of the
subgradient system \eqref{subgrad} or of the gradient system
\eqref{diffeq}, in case $f$ is smooth.

The proximal algorithm has been introduced in \cite{Martinez}. We
refer to \cite{lemaire} for a geometrical interpretation and to
\cite{patrick} for extensions to nonconvex problems. Let us now
recall from \cite[Section~2.4]{BDLM} the following important facts
(we give the short proof for completeness):

\begin{lemma}
[Geometrical interpretation]\label{Lemma_proj}Let $\{x_{i}\}_{i \geq
0}$ be a proximal sequence for a convex function $f$ and set
$r_i:=f(x_i)$. Then:

\begin{itemize}
\item[(i)] $f(x_{i+1}) < f(x_i)$ if $x_{i+1} \neq x_i$ ;
\item[(ii)] the point $x_{i+1}$ is the shortest distance projection of $x_{i}$ on the
sublevel set $[f\leq r_{i+1}]$.
\end{itemize}
\end{lemma}

\begin{proof} Assertion (i) is obvious, since $x_{i+1}$ minimizes
\eqref{1r}, thus for $x=x_{i}$ we get
\[
f(x_{i+1})+\frac{1}{2t_{i}}\|x_{i+1}-x_{i}\|^{2}\leq f(x_{i}).
\]
Assertion (ii) follows again from \eqref{1r}, by considering $x\in\lbrack
f\leq r_{i+1}].$ Indeed:
\[
f(x_{i+1})+\frac{1}{2t_{i}}\|x_{i+1}-x_i\|^{2}\leq
f(x)+\frac{1}{2t_{i}} \|x-x_i\|^{2}\leq
r_{i+1}+\frac{1}{2t_{i}}\|x-x_i\|^{2}.
\]
Since $r_{i+1}=f(x_{i+1})$ the above yields
\[
\|x_{i+1}-x_i\|\leq\|x-x_i\|,\text{ for all }x\in\lbrack f\leq r_{i+1}].
\]
The proof is complete.
\end{proof}

We obtain the following interesting consequence.

\begin{proposition}
[Self-contractedness of the proximal algorithm]\label{Prop_sc}Let
$\{x_{i}\}_{i \geq 0}$ be a proximal sequence for a convex function
$f$. Then the polygonal curve
\begin{equation}
P=\bigcup\limits_{i\in\mathbb{N}}[x_{i},x_{i+1}]\label{poly}
\end{equation}
is a self-contracted curve.
\end{proposition}

\begin{proof} Let $i_{1}<i_{2}<i_{3}.$ Notice first that for every $x\in\lbrack f\leq r_{m+1}],$ the function
\[
t\mapsto \|x_{m+1}+t(x_{m}-x_{m+1}) -x\|
\]
is increasing because the sublevel sets are convex.  Now Lemma~\ref{Lemma_proj} yields that for all $i_{1}\leq m<i_{2}$,  $x_{i_{3}}\in\lbrack f\leq
f(x_{m+1})]$  and
\[
\|x_{m+1}-x_{i_{3}}\|\leq\|x_{m}-x_{i_{3}}\|,
\]
which inductively yields
\begin{equation}
\|x_{i_{2}}-x_{i_{3}}\|\leq\|x_{i_{1}}-x_{i_{3}}\|. \label{forGuy}
\end{equation}
A simple argument now shows that
the polygonal curve is self-contracted.
\end{proof}

Combining the above with the length bound given in \eqref{mp-1} (in
view of Remark~\ref{Rem_improve}) we obtain the following
convergence result.

\begin{theorem}
[Convergence of the proximal algorithm]\label{thm_prox}Let
$f:\mathbb{R} ^{n}\rightarrow\mathbb{R}$ be convex and bounded from
below. Let $x_{0} \in\mathbb{R}^{n}$ be an initial point for the
proximal algorithm with parameters $\{t_{i}\}_{i}\subset(0,1]$. Then
the proximal sequence $\{x_{i}\}_{i}$ converges to some point
$x_{\infty}\in\mathbb{R}^{n}$ with a fast rate of convergence. That
is, there exists $c>0$, depending only on the dimension,
such that
\[
\sum_{i \in\mathbb{N}}\|x_{i}-x_{i+1}\|\leq c\|x_{0}-x_{\infty}\|.
\]

\end{theorem}

\begin{proof} By Proposition~\ref{Prop_sc} the polygonal
curve $P$ defined in \eqref{poly} is self-contracted. Consequently,
by Remark~\ref{Rem_improve} we know that \eqref{mp-1} applies,
yielding
\[
\ell(P)=\sum_{n\in \mathbb{N}}\|x_{i}-x_{i+1}\|\leq C_{n}\,\mathrm{diam\,}K,
\]
where $K$ is any compact convex set containing $\{x_{i}\}_{i}$. In addition, letting $i_1=0$ and $i_3\to +\infty$ in \eqref{forGuy} we   get
$$\|x_i-x_\infty \|\leq \|x_0-x_\infty\|,$$
for all $i \in \mathbb{N}$, which implies that
$\{x_{i}\}_{i}\subset
\overline{B}(x_{\infty},R)$ with $R=\|x_0-x_\infty\|$ and the result follows for $c=2C_{n}$.
\end{proof}

Theorem~\ref{thm_prox} guarantees the fast convergence of the
proximal sequence $\{x_{i}\}_{i}$ towards a limit point, with a
convergence rate independent of the choice of the parameters
$\{t_{i}\}_{i}.$ Notice however that for an arbitrary choice of
parameters, the limit point might not be a critical point of $f$ as
shows the following example:

\begin{example}
[Convergence to a noncritical point]Consider the (coercive, $C^{1}$) convex
function $f:\mathbb{R}\rightarrow\mathbb{R}$ defined by
\[
f(x)=\left\{
\begin{array}
[c]{cc}
x^{2}, & |x|\leq1/2\\
|x|-\frac{1}{4}, & |x|\geq1/2
\end{array}
\right.
\]
Then the proximal algorithm $\{x_{i}\}_{i}$, initialized at the point
$x_{0}=2$ and corresponding to the parameters $t_{i}=1/2^{i+1}$ ,
$i\in\mathbb{N}$, converges to the noncritical point $\bar{x}=1.$
\end{example}

In the above example, the choice of the parameters $\{t_{i}\}$ in
the proximal algorithm has the drawback that
$\sum_{i}t_{i}<+\infty$. In practice the choice of the step-size
parameters $\{t_{i}\}_{i}$ is crucial to obtain the convergence of
the sequence $\{f(x_{i})\}_{i}$ towards a \emph{critical value}; a
standard choice for this is any sequence satisfying $\sum
t_{i}=+\infty$, see for instance \cite{Guler}.

-------------------------------------------------------------

\noindent Aris Daniilidis

\smallskip

\noindent Departament de Matem\`{a}tiques, C1/308\newline Universitat
Aut\`{o}noma de Barcelona\newline E-08193 Bellaterra, Spain

\smallskip

\noindent E-mail: \texttt{arisd@mat.uab.es} \newline\noindent
\texttt{http://mat.uab.es/\symbol{126}arisd}

\smallskip

\noindent Research supported by the grant MTM2011-29064-C01 (Spain).

\bigskip

\noindent Guy David

\smallskip

\noindent Laboratoire de math\'{e}matiques, UMR 8628, B\^{a}timent 425 \newline
 Universit\'{e}
Paris-Sud, F-91405 Orsay Cedex, France \newline et Institut
Universitaire de France

\smallskip

\noindent E-mail: \texttt{Guy.David@math.u-psud.fr}\newline\noindent
\texttt{http://www.math.u-psud.fr/\symbol{126}gdavid/}

\bigskip

\noindent Estibalitz Durand-Cartagena

\smallskip

\noindent Departamento de Matem\'{a}tica Aplicada\newline ETSI
Industriales, UNED\newline Juan del Rosal 12, Ciudad Universitaria,
E-28040 Madrid, Spain

\smallskip

\noindent E-mail: \texttt{edurand@ind.uned.es}\newline\noindent
\texttt{https://dl.dropbox.com/u/20498057/Esti/Home\_Page.html}

\bigskip

\noindent Antoine Lemenant

\smallskip

\noindent Laboratoire Jacques Louis Lions (CNRS UMR 7598)\newline
Universit\'{e} Paris 7 (Denis Diderot)\newline175, rue du
Chevaleret, F-75013 Paris, France

\smallskip

\noindent E-mail: \texttt{lemenant@ljll.univ-paris-diderot.fr}\\
\noindent\texttt{http://www.ann.jussieu.fr/\symbol{126}lemenant/}

\end{document}